\documentclass[12pt, reqno]{amsart}
\pdfoutput=1
\usepackage{fullpage}

\usepackage{amssymb}
\usepackage{verbatim}
\usepackage{calrsfs}
\usepackage{graphicx}
\usepackage{epsfig}
\usepackage{color}
\usepackage[all]{xy}

\newtheorem{thm}{Theorem}[section]
\newtheorem{lem}[thm]{Lemma}
\newtheorem{cor}[thm]{Corollary}
\newtheorem{prop}[thm]{Proposition}

\newtheorem{conjecture}[thm]{Conjecture}
\theoremstyle{definition} 
\theoremstyle{remark} 
\theoremstyle{remark} 
\theoremstyle{remark} 
\theoremstyle{remark} 
\theoremstyle{definition} 
\theoremstyle{remark}  
\theoremstyle{definition}

\newcommand{\LL}{\mathcal{L}}     
\newcommand{\MM}{\mathcal{M}}  
\newcommand{\OO}{\mathcal{O}}    
\newcommand{\FF}{\mathbb{F}}      
\newcommand{\RR}{\mathbb{R}}     
\newcommand{\PP}{\mathbb{P}}      
\newcommand{\XX}{\mathcal{X}}      
\newcommand{\QQ}{\mathbb{Q}}      
\newcommand{\CC}{\mathbb{C}}      
\newcommand{\mm}{\mathfrak{m}}   
\newcommand{\Gm}{\mathbb{G}_m}  

\newcommand{\be}{\begin{equation}}
\newcommand{\ee}{\end{equation}}
\newcommand{\benn}{\begin{equation*}}
\newcommand{\eenn}{\end{equation*}}
\newcommand{\ba}{\begin{aligned}}
\newcommand{\ea}{\end{aligned}}
\newcommand{\bbm}{\begin{bmatrix}}
\newcommand{\ebm}{\end{bmatrix}}
\newcommand{\bpm}{\begin{pmatrix}}
\newcommand{\epm}{\end{pmatrix}}
\newcommand{\bi}{\begin{itemize}}
\newcommand{\ei}{\end{itemize}}
 
\newcommand{\xhelp}[1]{\textbf{$\diamondsuit\diamondsuit$ Fix me: #1}}

\newcommand{\Spec}{\operatorname{Spec}}
\newcommand{\ord}{\operatorname{ord}}
\newcommand{\Div}{\operatorname{div}}    
\newcommand{\supp}{\operatorname{supp}}   
\newcommand{\Pic}{\operatorname{Pic}}
\newcommand{\pr}{\operatorname{pr}}     
\newcommand{\sep}[1]{{#1}^{\operatorname{s}}}    
\newcommand{\Spf}{\operatorname{Spf}}    
\newcommand{\Frac}{\operatorname{Frac}}    
\newcommand{\chern}[1]{c_1\left(#1\right)}   
\newcommand{\codim}{\operatorname{codim}}  
\newcommand{\dist}{\operatorname{dist}}   
\newcommand{\an}[1]{#1^{\operatorname{an}}}  

\newcommand{\simarrow}{\stackrel{\sim}{\rightarrow}}    
\newcommand{\into}{\hookrightarrow}     
\newcommand{\alg}[1]{\overline{#1}}   

\theoremstyle{Theorem} \newtheorem*{siu}{Siu's Theorem}

\newcommand{\met}[1]{\overline{#1}}   
\newcommand{\achern}[1]{\widehat{c}_1\left(#1\right)}  
\newcommand{\lift}[1]{\widetilde{#1}}  
\newcommand{\fX}{\mathfrak{X}}    
\newcommand{\fL}{\mathfrak{L}}     
\newcommand{\metric}{\|\cdot\|}    
\newcommand{\orb}{O_v}    
\newcommand{\muinv}{\mu_{\varphi, v}}   
\newcommand{\Kvcirc}{K_v^{\circ}}   
\newcommand{\BB}{\mathcal{B}}
\newcommand{\YY}{\mathcal{Y}}
\newcommand{\UU}{\mathcal{U}}

\title[Dynamical Equidistribution]{
Equidistribution of dynamically small 
subvarieties over the function field 
of a curve 
}

\author{X.W.C. Faber}
\address{X.W.C. Faber\\
Department of Mathematics and Statistics\\
McGill University\\
805 Sherbrooke Street West \\
Montreal, QC  H2M 1X9 \\
CANADA} 
\email{xander@math.mcgill.ca}


\begin{document}

 \baselineskip=17pt 

	\begin{abstract}
		For a projective variety $X$ defined over a field $K$, there is a special class of morphisms 
		$\varphi: X \to X$ called algebraic dynamical systems. In this paper we take $K$
		to be the function field of a smooth curve and prove that at each place $v$ of $K$, subvarieties of
		dynamically small height are equidistributed on the associated Berkovich analytic space $\an{X}_v$. 
		We carefully develop all of the arithmetic intersection theory needed to state and prove this 
		theorem, and we present several applications on the non-Zariski density of preperiodic points and of 
		points of small height in field extensions of bounded degree. 
	\end{abstract}

\maketitle

\tableofcontents

 \setlength{\parskip}{3pt}


\section{Introduction}

	A typical arithmetic equidistribution theorem says something like the following: if $X$ is an algebraic
	variety over a global field $K$ that possesses an arithmetic height function $h$, then for each place
	$v$ of $K$ there exists an analytic space $\an{X}_v$ and a measure $\mu_v$ supported on it so 
	that the Galois orbits of any suitably generic sequence of algebraic points with heights tending to 
	zero become equidistributed with respect to $\mu_v$. The primary goal of this article is to prove 
	such a theorem when $K$ is the function field of a smooth curve and $h$ and $\mu_v$ are 
	intimately connected with the dynamics of a polarized endomorphism of $X$. Arithmetic 
	equidistribution theorems of this type abound in the literature of the last decade, beginning with the
	pioneering work of Szpiro/Ullmo/Zhang \cite{Szpiro-Ullmo-Zhang}. For a nice survey, 
	see the recent text by Silverman \cite[Chapter 3.10]{Silverman_Dynamics_Book_2007}.
		
		
	A \textit{(polarized algebraic) dynamical system} $(X, \varphi, L)$ defined over a field $K$ 
	consists of a projective variety $X / K$, an endomorphism $\varphi$ of $X$, and an ample line bundle $L$ on 
	$X$ such that $\varphi^*L \cong L^q$ for some integer $q \geq 2$. The line bundle $L$ is often 
	referred to as a \textit{polarization} of the dynamical system. A preperiodic point $x \in X$ is a closed 
	point such that the forward orbit $\{\varphi^n(x): n=1, 2, \ldots \}$ is finite. In studying the 
	arithmetic of algebraic dynamical systems, one might ask if, for example, the set of $K$-rational 
	preperiodic points is Zariski dense in $X$.  We present a criterion in \S\ref{Section: Applications}  
	for non-Zariski density as a corollary to the main equidistribution theorem of this paper. 
	
	Perhaps the most well-known example of an algebraic dynamical system is $(A, [2], L)$, where $A$ is 
	an abelian variety over $K$, $[2]$ is the multiplication by~2 morphism, and $L$ is a symmetric ample 
	line bundle on $A$ (i.e., $[-1]^*L \cong L$). In this case, torsion points of $A$ are the same as preperiodic 
	points for the morphism $[2]$. Another important example of a dynamical system is 
	$(\PP^d_K, \varphi, O(1))$ for some finite endomorphism $\varphi$ of $\PP^d_K$. 
	If $(z_0: \cdots :z_d)$ are homogeneous coordinates on projective space, then $\varphi$ can be 
	described more concretely by giving $d+1$ homogeneous polynomials $f_i$ of degree $q$ without a
	common zero over $\alg{K}$ and setting $\varphi = (f_0: \cdots: f_d)$. This is in a sense a ``universal dynamical 
	system'' because any other dynamical system can be embedded in one of this
	type. (See the paper of Fakhruddin \cite{Fakhruddin_Self_Maps_2003} for an explanation.) 
	
	To state a version of our theorem, we set the following notation. Let $\BB$ be a proper smooth 
	geometrically connected curve over a field $k$. Let $K=k(\BB)$ be the function field of $\BB$. 
	The places of $K$ --- i.e. equivalence classes of non-trivial valuations on $K$ that are trivial on $k$ --- 
	are in bijective correspondence with the closed points of $\BB$. For a place $v$,
	denote by $K_v$ the completion of $K$ with respect to $v$. 
	The Berkovich analytification of the scheme $X_{K_v}$ will be 
	denoted by $\an{X}_v$. It is a compact Hausdorff topological space with the same number of
	connected components as $X_{K_v}$. For any dynamical system $(X, \varphi, L)$, 
	there exists a measure $\muinv$ on $\an{X}_v$ that is invariant under
	$\varphi$; it reflects the distribution of preperiodic points for the morphism $\varphi$. 	
	Also, each closed point $x$ of $X$ breaks up into a finite set of points $\orb(x)$ in $\an{X}_v$, and
	we can define local degrees $\deg_v(y)$ for each $y \in \orb(x)$ such that 
	$\sum_{y \in \orb(x)} \deg_v(y) = \deg(x)$. To the set $\orb(x)$, we can associate a
	probability measure on $\an{X}_v$ by
		\[
			\frac{1}{\deg(x)}\sum_{y \in \orb(x)} \deg_v(y) \delta_y,
		\]
	There also exists a dynamical height function $h_{\varphi}$ on closed points of $X$ 
	defined via intersection theory, so that 
	$h_\varphi(x)$ measures the ``arithmetic and dynamical complexity'' of the point $x$. For example, 
	$h_{\varphi}(x) = 0$ if $x$ is a preperiodic point for $\varphi$. 
	

\begin{thm} \label{Theorem: Generic Equidistribution}
	Let $(X, \varphi, L)$ be an algebraic dynamical system over the function field $K$. Suppose 
	$\{x_n\}$ is a sequence of closed points of $X$ such that
		\begin{itemize}
			\item No infinite subsequence of $\{x_n\}$ is contained in a proper closed subset 
				of $X$, and
			\item $h_{\varphi}(x_n) \to 0$ as $n \to \infty$.
		\end{itemize} 
	Then for any place $v$ of $K$, and for any continuous function $f: \an{X_v} \to \RR$,
    	we have
      		\benn
         		\lim_{n \to \infty} \frac{1}{\deg(x_n)}\sum_{y \in \orb(x)} \deg_v(y) f(y)  =
         		\int_{\an{X_v}} f d\muinv.
      		\eenn
	That is, the sequence of measures 
	$\left\{\frac{1}{\deg(x_n)}\sum_{y \in \orb(x_n)} \deg_v(y) \delta_y\right\}_n$ 
	converges weakly to $\muinv$. 
\end{thm}

	Theorem~\ref{Theorem: Generic Equidistribution} is a simplified form of Theorem~\ref{Thm: 
	Stronger equidistribution}, which is the main result of this paper. In the latter result, we prove
	that generic nets of small subvarieties are equidistributed. See section~\ref{Section: Proof of 
	equidistribution} for the statement. The proof follows the technique  
	espoused in the work of Yuan \cite{Xinyi_Arithmetic_Bigness_2006} using the ``classical''
	version of Siu's Theorem from algebraic geometry.
		
	Several equidistribution theorems in the function field case already exist in the literature. 
	The work of Baker/Hsia \cite{Baker-Hsia_Dynamics} and
	Baker/Rumely \cite{Baker_Rumely_Equidistribution_2005}
	deal with polynomial maps and rational maps on the projective line, respectively. These papers approach 
	equidistribution using nonarchimedean capacity and potential theory. Favre/Rivera-Letelier
	give a quantitative form of equidistribution for rational maps on the projective line, which should
	apply to the function field case despite the fact that they state their results for number fields
	\cite{Favre-Rivera-Letelier_Quant_Equi_2006, Favre-Rivera-Letelier_Quant_Equi_Corrigendum}. 
	Petsche has proved a quantitative equidistribution result for points of small N\'eron-Tate height on elliptic curves
	using Fourier analysis \cite{Petsche_Elliptic_Equi_2007}. Gubler has approached equidistribution using 
	arithmetic intersection techniques (following Szpiro/Ullmo/Zhang). He gives an 
	equidistribution result for abelian varieties that admit a place of totally degenerate reduction 
	\cite{Gubler_Bogomolov_2007}. He has recently and independently proved equidistribution results essentially 
	equivalent to our Theorems~\ref{Theorem: Generic Equidistribution} 
	and~\ref{Thm: Stronger equidistribution} \cite{Gubler_FF_Equi_2008}. The work of Petsche and Gubler holds
	also over the function field of a smooth projective variety equipped with an ample divisor class.

	 In section~\ref{Arithmetic intersection theory over function fields} we state and prove all of 
	the tools of arithmetic intersection theory over function fields necessary for the task at hand. This 
	will only require the use of intersection theory of first chern classes in classical algebraic geometry
	and a small amount of formal and analytic geometry.
	Our approach is novel in that it is developed from global rather than local intersection theory.
	The associated measures and height functions will be introduced in this section as well.
	Section~\ref{Section: Dynamics} is devoted to recalling the construction of invariant metrics for 
	dynamical systems, the dynamical height $h_\varphi$, and the associated invariant measure 
	$\muinv$. We state and prove the main equidistribution theorem in 
	section~\ref{Section: Proof of equidistribution}. Finally, in section~\ref{Section: Applications} 
	we discuss applications of the equidistribution theorem to the distribution of preperiodic points 
	and to the Zariski density of dynamically small points. Section~\ref{Section: Appendix} contains 
	several auxiliary results whose proofs could not be located in the literature.
	 
	\textit{Acknowledgments.}
	I would like to express my gratitude to my advisor Shou-wu Zhang for explaining the big 
	ideas of arithmetic intersection theory to me, to Matt Baker for introducing me to algebraic
	dynamics, and to Xinyi Yuan for the many hours we've passed talking about Arakelov theory. 
	Thanks also go to Brian Conrad and Antoine Ducros for lending their expertise on nonarchimedean
	geometry.
	

\section{Arithmetic Intersection Theory over Function Fields}
\label{Arithmetic intersection theory over function fields}

	In this section we collect some definitions and facts from intersection theory needed for the 
	calculation of heights and for the proof of the main theorem. The basic principle is that heights can 
	be computed via (limits of) classical intersection numbers on models. Much has been written in the 
	literature on the subject of local intersection theory.
	Bloch/Gillet/Soul\'e have studied nonarchimedean local intersection theory by formally defining arithmetic 
	Chow groups using cycles on special fibers of models \cite{Bloch-Gillet-Soule_Arakelov_Theory}. 
	Gubler carries out a very careful study of local 
	intersection theory on special fibers of formal schemes in a more geometric fashion, and then he 
	patches it together into a global theory of heights using his theory of $M$-fields 
	\cite{Gubler_Heights_M_Fields_1997, Gubler_Local_Heights_of_Subvarieties_1998,
	Gubler_Local_Canonical_Heights_2003}. One can avoid most of the technical 
	difficulties that arise in these approaches by working almost exclusively with classical
	intersection theory on models of an algebraic variety; this requires a lemma of Yuan to
	lift data from a local model to a global model (Lemma~\ref{Lemma: Lifting Model Functions}).
	The upshot is that this treatment uses only ideas from algebraic geometry and a small input from 
	formal geometry, and it works for any field of constants $k$. 
	
	Most of these ideas appear in the literature on adelic metrics and height theory. We have endeavored
	to give complete proofs when the literature does not provide one or when the known proof for the number 
	field case requires significant modification.
	
\subsection{Notation and Terminology}
	\subsubsection{}
	When we speak of a variety, we will always mean an integral scheme, separated and of finite 
	type over a field. We do not require a variety to be geometrically integral. Throughout, we fix a 
	field of constants $k$ and a proper smooth geometrically connected $k$-curve $\BB$. At no point do we 
	require that the constant field $k$ be algebraically closed. Denote by $K$ the field of rational functions on $\BB$.
	
	\subsubsection{}
	The field $K$ admits a set of nontrivial normalized valuations that correspond bijectively to the
	closed points of $\BB$. The correspondence associates to a point $v \in \BB$ the valuation $\ord_v$ on the 
	local ring $\OO_{\BB, v}$ at $v$, which is a discrete valuation ring by smoothness. The normalization of $v$
	is such that a uniformizer in $\OO_{\BB,v}$ has valuation~1.  Extend $\ord_v$ to $K$ by additivity. 
	Such a valuation will be called a place of $K$. We will often identify closed points of $\BB$ with 
	places of $K$ without comment unless further clarification is necessary. 
	
	The places of $K$ satisfy a product formula. For $a \in K$, set $|a|_v = \exp\{-\ord_v(a)\}$. Then for any 
	$a \in K^{\times}$, we have
		\[
			\prod_{v \in \BB} |a|_v^{[k(v):k]} = 1,
		\]
	where $k(v)$ is the residue degree of the point $v$. This formula appears more frequently
	in its logarithmic form, in which it asserts that a rational function has the same number of zeros
	as poles when counted with the appropriate weights:
		\benn
			\sum_{v \in \BB} [k(v):k] \ord_v(a) = 0.
		\eenn
	
	\subsubsection{}
	Now let $X$ be a projective variety over $K$.  Given an open subvariety $\UU \subset \BB$, a 
	\textit{$\UU$-model of $X$} consists of the data of a $k$-variety $\XX$, a projective flat 
	$k$-morphism $\XX \to \UU$, and a preferred $K$-isomorphism $\iota: X \simarrow \XX_K$.  In most 
	cases the morphism $\XX \to \UU$ and the isomorphism $\iota$ will be implicit, and we will use them
	to identify $X$ with the generic fiber of $\XX$. If $L$ is a line bundle on $X$, a \textit{$\UU$-model of the 
	pair $(X,L)$} is a pair $(\XX,\LL)$ such that $\XX$ is a $\UU$-model of $X$ and $\LL$ is a line bundle on 
	$\XX$ equipped with a preferred isomorphism $\iota^*\LL|_{\XX_K} \simarrow L$.\footnote{As is customary,
	 we will use the terms ``line bundle'' and ``invertible sheaf'' interchangeably, but we will always work with them 
	as sheaves.} then a \textit{$\UU$-model of the pair $(X,L)$} Again, this 
	isomorphism will often be implicit. We may also say that $\LL$ is a 
	\textit{model of $L$}. To avoid trivialities, if we speak of a $\UU$-model $(\XX, \LL)$ of a pair $(X, L^e)$
	 without extra qualifier, we will implicitly assume that $e \geq 1$. 
	
	For a line bundle $\LL$ on a variety $\XX / k$ with function field $k(\XX)$, a \textit{rational section} of 
	$\LL$ is a global section of the sheaf $\LL \otimes k(\XX)$. Equivalently, a rational section is a choice 
	of a nonempty open set $U \subset \XX$ and a section $s \in \LL(U)$.
	
	Asking $X$ to be projective guarantees the existence of $\BB$-models. For example, let $L$ be a very ample
	line bundle on $X$. Let $X \hookrightarrow \PP^n_K \hookrightarrow \PP^n_{\BB} = \PP^n
	\times \BB$ be an embedding induced by $L$ followed by identifying $\PP^n_K$ with the 
	generic fiber of $\PP^n_{\BB}$, and define $\XX$ to be the Zariski closure of $X$ in $\PP^n_{\BB}$ 
	with the reduced subscheme structure. Let $\LL = \OO_{\XX}(1)$. Then $(\XX, \LL)$ 
	is a $\BB$-model of $(X, L)$. More generally, if $M$ is any line bundle on $X$, then we can write 
	$M = M_1 \otimes M_2^{\vee}$ for some choice of very ample line bundles $M_1$ and $M_2$. The 
	procedure just described gives a means of constructing $\BB$-models $(\XX_i, \MM_i)$ of $(X, M_i)$.
	 By the Simultaneous Model Lemma (Lemma~\ref{Lemma: Simultaneous Model Lemma}), there exists
	a single $\BB$-model $\XX$ of $X$ as well as line bundles $\MM'_i$ on $\XX$ such that $(\XX, \MM'_i)$
	is a $\BB$-model of $(X, M_i)$. Evidently $(\XX, \MM_1' \otimes (\MM_2')^{\vee})$ is a $\BB$-model
	of $(X, M)$. 

	\subsubsection{}
	For each place $v$ of $K$, we let $K_v$ be the completion of $K$ with respect to the valuation $v$.
	Let $\Kvcirc$ be the valuation ring of $K_v$. 
	
	We now briefly recall the relevant definitions from formal and analytic geometry. For full references
	on these topics, see \cite[\S2, 3.4]{Berkovich_Spectral_Theory_1990}, 
	\cite[\S1]{Berkovich_Vanishing_Cycles_Formal} and \cite{Bosch-Lutkebohmert_Formal_Rigid_1}. 
	For each place $v$ of $K$, let $\an{X}_v$ be the Berkovich analytic space associated to the 
	$K_v$-scheme $X_{K_v}$. It is a compact Hausdorff topological space equipped with the structure of 
	a locally ringed space. The space $\an{X}_v$ is covered by compact subsets of the form $\MM(\mathcal{A})$,
	where $\mathcal{A}$ is a strictly affinoid $K_v$-algebra and $\MM(\mathcal{A})$ is its Berkovich
	spectrum. As a set, $\MM(\mathcal{A})$ consists of all bounded multiplicative seminorms on $\mathcal{A}$.
	The definition of the analytic space $\an{X}_v$ includes a natural surjective morphism of locally ringed spaces 
	$\psi_v: \an{X}_v \to X_{K_v}$. If $L$ is a line bundle on $X$, 
	there is a functorially associated line bundle $L_v$ on $\an{X}_v$ defined by
	$\psi_v^*(L \otimes_K K_v)$. 
	
	A continuous metric on $L_v$, denoted $\|\cdot \|$, is a choice of a $K_v$-norm on each fiber of 
	$L_v$ that varies continuously on $\an{X}_v$. More precisely, suppose $\{(U_i, s_i)\}$ is a 
	trivialization of $L_v$ where $\{U_i\}$ is an open cover of $\an{X}_v$ and $s_i$ is a 
	generator of the $\OO_{\an{X}_v}(U_i)$-module $L_v(U_i)$. Then the metric $\|\cdot\|$ is defined
	by a collection of continuous functions $\rho_i: U_i \to \RR_{>0}$ satisfying an appropriate cocycle
	condition by the formula $\|s(x)\| = |\sigma(x)| \rho(x)$, where $s$ is any section of $L_v$ over 
	$U_i$ and $\sigma$ is the regular function on $U_i$ such that $s = \sigma s_i$. Here $|\sigma(x)|$ denotes 
	the value of the seminorm corresponding to the point $x$ at the function $\sigma$.\footnote{
			Several authors (e.g., \cite{Bombieri-Gubler_2006}) speak of \textit{bounded} 
			continuous metrics on $L \otimes \CC_v$ over $X_{\CC_v}$, where $\CC_v$
			is the completion of an algebraic closure of $K_v$. 
			Bounded and continuous metrics in their context are equivalent to our notion of
			continuous metric by compactness of the Berkovich analytic space $\an{X}_v$.
		}
	A \textit{formal metric} $\|\cdot\|$ on the line bundle $L_v$ is one for which there exists a 
	trivialization $\{(U_i, s_i)\}$ so that metric is defined by $\rho_i \equiv 1$. Equivalently, for this
	trivialization we have $\|s(x)\| = |\sigma(x)|$ with $s$ and $\sigma$ as above. 
		
	To an admissible formal $K_v^{\circ}$-scheme $\fX_v$, one can associate in a functorial way 
	its \textit{generic fiber} $\fX_{v, \eta}$, which is an analytic space in the sense of Berkovich. For us, 
	the most important example will be when $\fX_v$ is the formal completion of a proper flat 
	$\Kvcirc$-scheme with generic fiber $X / K_v$. In this setting, properness implies there is a 
	canonical isomorphism $\fX_{v, \eta} \simarrow \an{X}_v$
	\cite[Thm.~A.3.1]{Conrad_Irreducible_Components}. A formal line bundle $\fL_v$ on 
	 $\fX_v$ determines a formal metric on $L_v = \fL_v \otimes K_v$ as follows. 
	Let $\{(\mathfrak{U}_i, s_i)\}$ be a formal trivialization of $\fL_v$. Then the generic fiber functor 
	gives a trivialization of $L_v$, namely $\{(\mathfrak{U}_{i, \eta}, s_i)\}$. We set $\rho_i \equiv 1$. This 
	definition works (i.e., the functions $\rho_i$ transform correctly) because the transition functions 
	for the cover $\{(\mathfrak{U}_{i, \eta}, s_i)\}$ all have supremum norm~$1$
	\cite[\S7]{Gubler_Local_Heights_of_Subvarieties_1998}. 
	
	For an open subscheme $\UU \subset \BB$, a $\UU$-model $(\XX,\LL)$ of  the pair $(X,L^{e})$ 
	determines a family of continuous metrics, one on $L_v$ for each place $v$ of $\UU$. Indeed, let 
	$\widehat{\XX}_v$ be the formal completion of the scheme $\XX_{K_v^{\circ}}$ along its closed 
	fiber, and let $\widehat{\LL}_v$ be the formal completion of $\LL$. As $\XX_{K_v^{\circ}}$ is 
	$\Kvcirc$-flat, we know that $\widehat{\XX}_v$ is flat over $\Kvcirc$, hence admissible. 
	We denote by $\metric_{\LL,v}$ the formal metric on $\widehat{\LL}_v \otimes K_v \cong L_v^e$. 
	A metric $\metric_{\LL,v}^{1/e}$ on $L_v$ can then be given by defining 
	$\|\ell\|^{1/e}_{\LL,v} = \|\ell^{\otimes e}\|_{\LL,v}^{1/e}$ for any local section $\ell$ of $L_v$. 
	
	There is an important subtlety here that is worth mentioning. If $(\XX, \LL)$ is a $\BB$-model of
	$(X, O_X)$ and $v$ is a place of $\BB$, then we get a metric $\metric = \metric_{\LL, v}$ on 
	$O_{\an{X}_v}$. However, $O_X \simarrow O_X^n$ for any positive integer $n$ via the canonical isomorphism
	given by $1 \mapsto 1^{\otimes n}$, and so we just as easily view $(\XX, \LL)$ as  $\BB$-model of 
	$(X, O_X^n)$. The metric induced on $O_{\an{X}_v}$ by this $\BB$-model is given by $\metric^{1/n}$.

	\subsubsection{}
	An \textit{adelic metrized line bundle} $\met{L}$ on $X$ consists of the data of a line bundle 
	$L$ on $X$ and for each place $v \in \BB$ a continuous metric $\metric_{v}$ on 
	the analytic line bundle $L_v$ subject to the following adelic coherence condition:
	there exists an open subscheme $\UU \subset \BB$, and a $\UU$-model $(\XX, \LL)$ of the pair 
	$(X,L^e)$ for some positive integer $e$ such that at all places $v$ in $\UU$ we have equality of the 
	metrics  $\|\cdot\|_v = \|\cdot\|_{\LL,v}^{1/e}$. The adelic metrized line bundle $\met{L}$ will be 
	called \textit{semipositive} if there exists an open subscheme $\UU \subset \BB$, a sequence of
	 positive integers $e_n$ and a sequence of $\BB$-models $(\XX_n, \LL_n)$ 
	 of the pairs $(X,L^{ e_n})$ such that 
		\begin{itemize}
			\item $\LL_n$ is \textit{relatively semipositive} for all $n$: it has nonnegative degree on 
				any curve in a closed fiber of $\XX$;\footnote{Modern algebraic geometers
				would probably call this line bundle \textit{relatively nef}, but we preserve the 
				term ``relatively semipositive'' for historical reasons.}
			\item For each place $v$ of $\UU$, we have equality of the metrics
				 $\|\cdot\|_{v} = \|\cdot\|_{\LL_{n,v}}^{1/e_n}$ for all $n$; and
			\item For each place $v \not \in \UU$, the sequence of metrics
			$\|\cdot\|_{\LL_{n,v}}^{1/e_n}$ converges uniformly to $\|\cdot\|_{v}$ on $\an{X}_v$.
		\end{itemize}
	The distance between two metrics $\metric_{1,v}$ and $\metric_{2,v}$ on $L_v$ 
	is given by 
		\benn
			\dist_v(\metric_{1,v}, \metric_{2,v}) = \max_{x \in \an{X}_v} \left| \log 
			\frac{\|s(x)\|_{1,v}}{\|s(x)\|_{2,v}}\right|, 
		\eenn	
	where $s$ is any local section of $L_v$ that does not vanish at $x$.
	The quotient inside the logarithm is independent of the choice of $s$ and defines a non-vanishing continuous
	function on the compact space $\an{X}_v$. This implies the existence of the maximum.
	A sequence $(\metric_{n,v})_n$ of metrics on $L_v$ is said to converge uniformly to 
	the metric $\metric_{v}$ if $\dist_v(\metric_{n,v}, \metric_{v}) \to 0$ as $n \to \infty$. By abuse
	of notation, in the definition of semipositive metrized line bundle we may say that
	the $\BB$-models $\LL_n$ converge uniformly to $\met{L}$.

	An adelic metrized line bundle is called \textit{integrable} if it is of the form 
	$\met{L_1} \otimes \met{L_2}^{\vee}$ for two semipositive metrized line bundles 
	$\met{L_1}$ and $\met{L_2}$. We denote by $\overline{\Pic}(X)$ the group of 
	integrable adelic metrized line bundles on $X$ (under tensor product). Semipositive metrized line
	bundles form a semigroup inside $\overline{\Pic}(X)$. (This follows from 
	Lemma~\ref{Lemma: Simultaneous Model Lemma} and the fact that the tensor product of nef line
	bundles is nef.)

\subsection{Arithmetic Intersection Numbers} \label{Section: Intersection Numbers}

	Given $d+1$ integrable metrized line bundles $\met{L}_0, \ldots, \met{L}_d$, we wish to 
	define the arithmetic intersection number $\achern{\met{L}_0} \cdots \achern{\met{L}_d}$.
	In the present situation, this is accomplished by approximating the metrics
	on these line bundles using $\BB$-models, performing an intersection calculation on the 
	$\BB$-models, and then passing to a limit. This procedure appears in the work of Zhang
	for the number field setting \cite{Zhang_Small_Points_1995}, and we devote this section to
	proving it works very generally for function fields of transcendence degree one.
	
	Those well acquainted with intersection products in Arakelov theory will find little
	surprising in Theorem~\ref{Theorem: Intersection properties} below and probably nothing 
	new in its proof. However, the literature on the subject appears only to contain these facts and
	their proofs in the case where $K$ is a number field or when $K$ is a function field over
	an algebraically closed field $k$ \cite[\S3]{Gubler_Bogomolov_2007}. Here we have removed 
	the hypothesis that $k$ is algebraically closed, with the only consequence being that a residue 
	degree appears in some of the formulas. The benefit is that the proofs are global and algebraic in 
	nature rather than local and formal. 
	
	For a projective variety $Y$ over a field $k$, we will almost always identify a zero cycle 
	$\sum n_P [P]$ with its degree $\sum n_P[k(P):k]$. For a line bundle $L$ on $Y$, we write 
	$\chern{L}$ to denote the first Chern class valued in $A^1(Y)$, the group of codimension-$1$ cycles on $Y$.
	When $Y$ has dimension $d$, we will write 
	$\deg_{L_1, \ldots, L_d}(Y)$ or $\chern{L_1} \cdots \chern{L_d}$ instead of the more correct 
	form $\deg \left(\chern{L_1} \cdots \chern{L_d}\cdot[Y]\right)$.
	Similarly, we may write $\deg_{L}(Y)$ for $\chern{L}^d$.

\begin{thm} \label{Theorem: Intersection properties}
	Let $X$ be a projective variety of dimension $d$ over the function field $K$. 
	There exists a pairing $\achern{\met{L}_0} \cdots \achern{\met{L}_d}$ on 
	$\overline{\Pic}(X)^{d+1}$ satisfying the following properties:
		\begin{enumerate}
			\item Let $\XX$ be a $\BB$-model of $X$ and let $\LL_0, \ldots, \LL_d$ models on $\XX$ 
				of the line bundles $L_0^{e_0}, \ldots, L_d^{e_d}$ for some positive integers
				$e_0, \ldots, e_d$. If $\met{L}_0, \ldots, \met{L}_d$ are the associated adelic 
				metrized line bundles with underlying algebraic bundles $L_0, \ldots, L_d$,
				then
					\benn
						\achern{\met{L}_0} \cdots \achern{\met{L}_d} = 
						\frac{\chern{\LL_0} \cdots \chern{\LL_d}}{e_0\cdots e_d}.
					\eenn
					
			\item Let $\met{L}_1, \ldots, \met{L}_d$ be semipositive metrized line bundles and
				take  $\met{L}_0$ and $\met{L}_0'$ to be two integrable metrized lined bundles
				with the same underlying algebraic bundle $L_0$. The metrics of $\met{L}_0$
				and $\met{L}_0'$ agree at almost all places, and if $\metric_{0,v}$ and $\metric_{0,v}'$
				are the corresponding metrics at the place $v$, then 
					\benn
						\ba
							\Bigg|\achern{\met{L}_0 \otimes \left(\met{L}_0'\right)^{\vee}}  &
							\achern{\met{L}_1} \cdots \achern{\met{L}_d}\Bigg|   \\
							& \leq \deg_{L_1, \ldots, L_d}(X)\sum_{v \in \BB} [k(v):k] 
							\dist_v\left(\metric_{0,v}, \metric_{0,v}'\right).
						\ea
					\eenn
					
			\item $\achern{\met{L}_0} \cdots \achern{\met{L}_d}$ is symmetric and
				multilinear in $\met{L}_0, \ldots, \met{L}_d$.
			
			\item Let $Y$ be a projective variety over $K$, and suppose $\varphi: Y \to X$ is a 
				generically finite surjective morphism. Then
					\benn
						\achern{\varphi^*\met{L}_0} \cdots \achern{\varphi^*\met{L}_d} = 
						\deg(\varphi) \ \achern{\met{L}_0} \cdots \achern{\met{L}_d}.
					\eenn
		\end{enumerate}
	Moreover, the pairing on $\overline{\Pic}(X)^{d+1}$ is uniquely defined by properties {\em (i)} and {\em (ii)}. 
\end{thm}
	
	The proof of this theorem will occupy the remainder of section~\ref{Section: Intersection Numbers}.


\subsubsection{Preliminary Lemmas}

	In order to define the pairing and show it is well-defined, we need a number of preliminary
	facts. The following result allows one to take metrics induced from line bundles
	on several different models of $X$ and consolidate the data on a single $\BB$-model of $X$. 
	
\begin{lem}[Simultaneous Model Lemma] \label{Lemma: Simultaneous Model Lemma}
	Let $(\XX_1,\LL_1), \ldots, (\XX_n,\LL_n)$ be $\BB$-models of the pairs 
	$(X,L_1), \ldots, (X,L_n)$, respectively. Then there exists a single
	$\BB$-model $\XX$ along with $\BB$-morphisms $\pr_i:\XX \to \XX_i$ that restrict to 
	isomorphisms on the generic fiber. For each $i$, the line bundle $\LL_i' = \pr_i^*\LL_i$ is a model
	of $L_i$ such that for each place $v$ of $K$, the corresponding metrics on $L_{i,v}$ satisfy 
	$\|\cdot\|_{\LL_{i,v}} = \|\cdot\|_{\LL_{i',v}}$. If $\LL_i$ is nef on $\XX_i$, then $\LL_i'$ is nef 
	on $\XX$.
\end{lem}

\begin{proof}

	Suppose that $n=2$; the general case is only notationally more difficult. 	
	Consider the following commutative diagram:
			\benn	
				\xymatrix{
				X \ar@/^1pc/ [drr] \ar[dr]^{\Delta} \ar@/_1pc/[ddr] & & \\
				&  \XX_1 \times_\BB \XX_2 \ar[r]^{ \ \ \pr_2} 
				\ar[d]^{\pr_1} & \XX_2 \ar[d] \\
				& \XX_1 \ar[r] & \BB \\
				}
			\eenn
	The maps $\XX_i \to \BB$ are the structure morphisms.  
	The map $\Delta$ is the diagonal embedding of $X$ in the generic fiber
	of $\XX_1 \times_\BB \XX_2$, which is just $X \times_K X$, and the square
	is the fiber product. Let $\XX$ be the Zariski closure of $\Delta(X)$ in 
	$\XX_1 \times_\BB \XX_2$ with the reduced subscheme structure; it is a $\BB$-model
	of $X$ via the diagonal map as preferred isomorphism on the generic fiber. When 
	restricted to $\Delta(X)$, the two projections $\pr_i$ are isomorphisms, and so they 
	induce isomorphisms between the generic fibers of $\XX$ and $\XX_i$.
	Define $\LL_i' = \pr_i^*\LL_i|_{\XX}$. 
		
	The metrics on a line bundle $\LL'$ are unchanged by pullback through a $\BB$-morphism 
	$\XX \to \XX'$ that restricts to an isomorphism on the generic fiber. Indeed, completing at the closed fiber
	over $v$ gives a morphism of admissible formal $\Kvcirc$-schemes $f: \fX \to \fX'$. Over open sets 
	$\Spf \mathcal{A}' \subset \fX'$ and $\Spf \mathcal{A} \subset f^{-1}(\Spf \mathcal{A}') \subset \fX$
	where the line bundles $\fL'$ and $f^*\fL'$ are trivial, flatness of these algebras 
	over $K_v^{\circ}$ implies that we have a commutative diagram of inclusions:
		\benn
				\xymatrix{
				\mathcal{A}' \ar[r]^{\alpha} \ar[d] & \mathcal{A} \ar[d]\\  
				\mathcal{A}' \otimes_{\Kvcirc} K_v \ar[r] &  \mathcal{A} \otimes_{K_v^{\circ}} K_v  \\
				}	
		\eenn
	 Over these open sets, the formal metrics on $\fL'$ and on $f^*\fL'$ are given at $x$ by 
	evaluation \cite[Lemma~7.4]{Gubler_Local_Heights_of_Subvarieties_1998}; 
	i.e., for any local section $s$ of $\fL'$ defined near $x$ corresponding to an element 
	$\sigma \in \mathcal{A}'$, we have
		\[
			\|f^*(s)(x)\|_{f^*\fL'} = |\alpha(\sigma)(x)| = |\sigma(x)| = \|s(x)\|_{\fL'},
		\]
	where the middle equality follows because $f$ is an isomorphism on the generic fiber.

	The final claim of the lemma is simply the fact that the pullback of a nef line bundle is nef
	(which follows from the projection formula for classical intersection products).
\end{proof}

	Next we show that intersection numbers vary nicely in fibers over the base curve $\BB$. This
	is well-known for fibers over the closed points of $\BB$ 
	(cf. \cite[\S10.2]{Fulton_Intersection_Theory_1998}, ``Conservation of Number''), but we are 
	also interested in comparing with the generic fiber. 
		
\begin{lem} \label{Lemma: Conservation of Number}
	Let $X$ be a projective variety over $K$ of dimension $d$, $L_1, \ldots, L_d$ line bundles on $X$, 
	$\pi:\XX \to \BB$ a $\BB$-model of $X$, and $\LL_1, \ldots, \LL_d$ models of $L_1, \ldots, L_d$ on
	$\XX$. Then for any closed point $v \in \BB$, we have the equality
		\benn
			[k(v):k]\deg_{L_1, \ldots, L_d}(X) = \chern{\LL_1}\cdots \chern{\LL_d}\cdot [\XX_v],
		\eenn
	where $\XX_v$ is the (scheme-theoretic) fiber over $v$.		
\end{lem}	

\begin{proof}
	As $\BB$ is regular, we may speak of its Cartier and Weil divisors interchangeably. Note that 
	$[\XX_v]$ is the cycle associated to the Cartier divisor $\pi^*[v]$. Indeed, each is cut out on
	$\XX$ by the image of a local equation for the point $v$ under the map $\OO_{\BB} \to \OO_{\XX}$.
		
	Now we proceed by induction on $d = \dim X$. If $d=0$, then $X = \Spec F$ corresponds to 
	a finite extension of fields $F / K$. Also, $\pi:\XX \to \BB$ is a proper surjection of curves. 
	Whence
		\benn
			\ba
				\deg([\XX_v]) = \deg(\pi^*[v]) &= \deg(\pi_*\pi^*[v]) \\
				&= [k(\XX):k(\BB)]\deg([v]) \qquad \text{ by the projection formula} \\
				&= [F:K][k(v):k], 
			\ea
		\eenn
	which is exactly what we want.
	
	Next assume the result holds for all $K$-varieties of dimension at most $d-1$, and let $\dim X = d$.
 	Let $s_d$ be a rational section of $\LL_d$. Write $[\Div(s_d)] =D_h + D_f$, where $D_h$ is 
	horizontal and $D_f$ is vertical on $\XX$. Then $D_f \cdot [\XX_v] = D_f \cdot [\pi^*[v]] = 0$ in 
	the Chow group because we may use linear equivalence on $\BB$ to push $\pi^*[v]$ away from the 
	support of $D_f$. In the next computation, we use the letter $\YY$ to denote an arbitrary horizontal
	prime divisor on $\XX$, and we set $Y = \YY_K$, the generic fiber of $\YY$. This gives a bijective 
	correspondence between horizontal prime divisors on $\XX$ and prime divisors on $X$. 
	Moreover $O_{X,Y} = \OO_{\XX, \YY}$ for any such prime divisor, and $\codim(Y,X) = \codim(\YY,\XX)$.
	Now we compute:
		\benn
			\ba
				\chern{\LL_1}\cdots\chern{\LL_d}\cdot[\XX_v] &= 
				\chern{\LL_1} \cdots \chern{\LL_{d-1}} \cdot[\pi^*[v]] \cdot [\Div(s_d)] \\
				&= \chern{\LL_1} \cdots \chern{\LL_{d-1}} \cdot[\pi^*[v]] \cdot D_h \\
				&= \sum_{\substack{
				{ \YY \subseteq \supp(D_h)} \\ \codim(\YY, \XX) = 1}}
				\ord_{\YY} (s_d) \chern{\LL_1|_{\YY}} \cdots 
				\chern{\LL_{d-1}|_{\YY}} \cdot[\pi^*[v]|_{\YY}] \\
				&= [k(v):k] \sum_{\substack{
				{ Y \subseteq \supp(D_h \cap {X})} \\ \codim(Y, X) = 1 }}
				\ord_{Y} (s_d|_X)\deg_{L_1|_Y, \ldots, L_{d-1}|_Y}
				(Y) \\
				&= [k(v):k] \chern{L_1} \cdots \chern{L_{d-1}}\cdot[\Div(s_d|_X)] \\
				&= [k(v):k] \deg_{L_1, \ldots, L_d}(X).	
			\ea
		\eenn
	In the third to last equality we applied the induction hypothesis. This completes the proof.  
\end{proof}

	The next lemma allows us to see what happens to metrics after pullback through a morphism of $\BB$-models. 

\begin{lem} \label{Lemma: Pullback of metrics}
	Let $\varphi: Y \to X$ be a morphism of projective $K$-varieties. Let $L$ be a line bundle
	on $X$ and $(\XX, \LL)$  a $\BB$-model of the pair $(X,L)$. 
		\begin{enumerate}
			\item There exists a $\BB$-model $\YY$ of $Y$ and a $\BB$-morphism $\lift{\varphi}: \YY \to \XX$
				 such that $\lift{\varphi}_K = \varphi$. 
				 
			\item For any $\BB$-morphism $\lift{\varphi}$ as in $\mathrm{(i)}$, 
				the pair $(\YY, \lift{\varphi}^*\LL)$ is a $\BB$-model of the pair
				 $(Y, \varphi^*L)$, and for each $v$ we have the equality of metrics 
				 $\varphi^*\metric_{\LL, v} = \metric_{\lift{\varphi}^*\LL, v}.$
		\end{enumerate}
\end{lem}

\begin{proof}
	Let $\YY_0$ be any $\BB$-model of $Y$. Define $j_X: X \hookrightarrow \XX$ to be the 
	inclusion of the generic fiber (always implicitly precomposed with the preferred isomorphism
	$\iota: X \simarrow \XX_K$), and similarly for $j_Y: Y \hookrightarrow \YY_0$. 
	Let $\Gamma_{\varphi}: Y \to Y \times_K X$ be the graph morphism. 
	Consider the commutative diagram
		\benn	
			\xymatrix{
				Y \ar@/^1pc/ [drr]^{j_X \circ \varphi} \ar[dr]^{\lift{\Gamma}_\varphi
			 	} \ar@/_1pc/[ddr]_{j_Y} & & \\
				&  \YY_0 \times_{\BB} \XX \ar[r]^{\ \ \pr_2} \ar[d]^{\pr_1} & \XX \ar[d] \\
				& \YY_0 \ar[r] & {\BB} \\
			}
		\eenn
	where $\lift{\Gamma}_\varphi = (j_Y \times j_X) \circ\Gamma_\varphi$. 
	Define $\YY$ to be the Zariski closure of 
	$\lift{\Gamma}_\varphi(Y)$ in $\YY_0 \times_{\BB} \XX$ with the 
	reduced subscheme structure, and give it the obvious structure as a flat proper $\BB$-scheme. 
	Let $\lift{\varphi} = \pr_2|_{\YY}$. The graph morphism
	$\Gamma_\varphi$ gives the preferred isomorphism between $Y$ and the generic fiber of $\YY$.
	With this identification, it is evident that $\lift{\varphi}_K = \varphi$.
	
	For any $\lift{\varphi}: \YY \to \XX$ such that $\lift{\varphi}_K = \varphi$, we find that 
	$(\YY, \lift{\varphi}^*\LL)$ is a $\BB$-model of $(Y, \varphi^*L)$ because passage to the generic
	fiber commutes with pullback of line bundles.
	
	The final statement of the lemma is a consequence of the compatibility of the generic fiber functor
	for admissible formal schemes and the pullback morphism. Let 
	$\widehat{\varphi}_v:\widehat{\XX}_v \to \widehat{\YY}_v$ be the morphism induced between the
	formal completions of $\XX$ and $\YY$, respectively, along their closed fibers over $v$, and let 
	$\widehat{\LL}_v$ be the formal completion of $\LL$. The metric $\metric_{\LL, v}$ is determined by
	a formal trivialization $\{(\mathfrak{U}_i, s_i)\}$ of $\widehat{\LL}_v$ and a collection of functions
	$\rho_i: \mathfrak{U}_{i, \eta} \to \RR_{>0}$. 
	(Recall that $\mathfrak{U}_{i, \eta}$ is the generic fiber of $\mathfrak{U}_i$, and that 
	$\{\mathfrak{U}_{i, \eta}\}$ is a cover of $\an{X}_v$.) If $\an{\varphi}_v: \an{Y}_v \to \an{X}_v$ is 
	the induced morphism between analytic spaces, we find that 
	$(\an{\varphi}_v)^{-1}(\mathfrak{U}_{i, \eta}) = \widehat{\varphi}_v^{-1}(\mathfrak{U}_i)_\eta$; i.e., 
	pullback commutes with formation of the generic fiber. Therefore both metrics 
	$\varphi^* \metric_{\LL, v}$ and $\metric_{\lift{\varphi}^*\LL, v}$ are given by the cover 
	$\{((\an{\varphi}_v)^{-1}(\mathfrak{U}_{i, \eta}), (\an{\varphi}_v)^*(s_i) )\}$ and the functions
	$\rho_i \circ \an{\varphi}_v : (\an{\varphi}_v)^{-1}(\mathfrak{U}_{i, \eta}) \to \RR_{>0}$.
\end{proof}

	Next we prove an estimate that indicates the dependence of intersection numbers for line
	bundles on the metrics induced by them. In the course of the proof we shall need to relate
	special values of the metric on a section of the model to the orders of vanishing of the section. 
	We recall now how this works. Let $(\XX, \LL)$ be a $\BB$-model of $(X,O_X)$ with $\XX$ normal, and let
	$s$ be a rational section of $\LL$ that restricts to the section $1$ on the generic fiber. Write 
	$[\Div(s)] = \sum_v \sum_j \ord_{W_{v,j}}(s) \  [W_{v,j}]$, where 
	$\{W_{v,j}\}_j$ are the distinct irreducible components of the fiber $\XX_v$ over the point $v$. 
	Let us also write $[\XX_v] = \sum_j m(v,j) [W_{v,j}]$ for some positive
	integers $m(v,j)$. 
	
	There exists a surjective reduction map $r: \an{X}_v \to \XX_v$, and for 
	each component $W_{v,j}$, there is a unique point $\xi_{v,j} \in \an{X}_v$ mapping to the 
	generic point of $W_{v,j}$. See \cite[\S1]{Berkovich_Vanishing_Cycles_Formal} for the construction
	of the reduction map and an argument that shows its image is closed. We will now give an argument
	to conclude that its image contains the generic points of $\XX_v$.  
	
	The fiber $\XX_v$ is unchanged if we replace $\XX$ by $\XX \times_{\BB} \Spec \OO_{\BB, v}$. 	
	Let $\eta_{v,j}$ be the generic point of $W_{v,j}$, and let $\Spec A$ be an affine open subscheme containing 
	$\eta_{v,j}$. Set $\hat{A}$ to be the $\mm_v$-adic completion of $A$,
	where $\mm_v$ is the maximal ideal of $\OO_{\BB,v}$, and set 
	$\mathcal{A} = \hat{A} \otimes_{\Kvcirc} K_v$ to be the corresponding strict $K_v$-affinoid algebra. 
	Define a multiplicative seminorm on $A$ by 
		\benn
			|a| = \begin{cases}
					\exp\left( -\ord_{W_{v,j}}(a) / m(v,j)\right), & a \not= 0 \\
					0, & a=0.
				\end{cases}
		\eenn
	Extending $|\cdot|$ to $\hat{A}$ by continuity and then to $\mathcal{A}$ by taking fractions gives a
	bounded multiplicative $K_v$-seminorm on $\mathcal{A}$; it corresponds to a point 
	$\xi_{v,j} \in \MM(\mathcal{A}) \subset \an{X}_v$ such that $r(\xi_{v,j}) = \eta_{v,j}$. Uniqueness
	follows from the fact that any $x \in r^{-1}(\eta_{v,j})$ induces a valuation on $\Frac(A)$ whose valuation
	ring dominates $\OO_{\XX, \eta_{v,j}}$. But by normality, the two valuation rings must coincide, and hence
	$x = \xi_{v,j}$.
	
	Finally note that if the rational section $s$ of $\LL$ corresponds to a rational function $\sigma \in \Frac(A)$, 
	then the definitions immediately imply that 
		\be \label{Equation: Model Function value versus order function}
			-\log \|1(\xi_{v,j})\|_{\LL,v}  = -\log |\sigma(\xi_{v,j})| 
			= \frac{\ord_{W_{v,j}}(s)}{m(v,j)}.
		\ee
	Compare with \cite[2.3]{Chambert-Loir_Measures_2005}.
	
\begin{lem} \label{Lemma: Metric bound on intersections}
	Let $\XX$ be a $\BB$-model of $X$, and suppose $\LL_1, \ldots, \LL_d$ on $\XX$ are relatively
	semipositive models of line bundles $L_1, \ldots, L_d$ on $X$. Let $L_0$ be another line bundle on
	$X$ and $\LL_0$ and $\LL_0'$ two models of $L_0$. Then the metrics on $L_0$ induced by 
	$\LL_0$ and $\LL_0'$ differ at only finitely many places $v$, and we have
		\benn
			\left|\chern{\LL_0 \otimes (\LL_0')^{\vee}} \chern{\LL_1}\cdots \chern{\LL_d} \right|
				\leq \deg_{L_1, \ldots, L_d}(X) \sum_{v \in \BB} [k(v):k] 
				\dist_v\left(\metric_{\LL_{0,v}}, \metric_{\LL_{0,v}'}\right).
		\eenn
\end{lem}	

\begin{proof}
	First note that since $\LL_0$ and $\LL_0'$ restrict to the same line bundle on the generic fiber, 
	they must be isomorphic over some nonempty open subset $\UU \subset \BB$. The places 
	corresponding to points of $\BB \setminus \UU$ are finite in number, and these are the only 
	places at which the metrics of $\LL_0$ and $\LL_0'$ can differ.
	
	We may reduce to the case that $\XX$ is normal. Indeed, let $\lift{\varphi}: \lift{\XX} \to \XX$ be the
	normalization morphism. Endow $\lift{\XX}$ with the structure of $\BB$-scheme via composition
	of $\lift{\varphi}$ with the structure morphism for $\XX$. Define $\varphi = \lift{\varphi}_K$. Then
	$\varphi$ and $\lift{\varphi}$ have degree~$1$, so the projection formula for classical intersection
	products shows that intersection numbers in the above inequality are unaffected by pullback to 
	$\lift{\XX}$ and $(\lift{\XX})_K$. To see that the distance between the metrics is unaffected, we 
	note that by Lemma~\ref{Lemma: Pullback of metrics}, $\metric_{\lift{\varphi}^*\LL, v} = 
	\varphi^*\metric_{\LL, v}$ for any line bundle $\LL$ on $\XX$. Since the morphism 
	$\an{\varphi}: \an{\lift{\XX}_{K_v}} \to \an{\XX_{K_v}}$ is surjective, we find
		\[
			\dist_v\left(\varphi^*\metric_{\LL_0, v}, \varphi^*\metric_{\LL_0', v}\right) =
			\dist_v\left(\metric_{\LL_0, v}, \metric_{\LL_0', v}\right).
		\]
	Thus we may replace $\XX$, $X$, and $\LL_i$ by 
	$\lift{\XX}$, $(\lift{\XX})_K$, and $\lift{\varphi}^*\LL_i$, respectively. Note that $\lift{\varphi}^*\LL_i$
	is relatively semipositive by the projection formula.

	Finally, it suffices to prove that if $\LL_0$ is any model of the trivial bundle on $X$, then
		\benn \label{Equation: Sufficient inequality}
			\left|\chern{\LL_0}\cdots \chern{\LL_d}\right|
				\leq \deg_{L_1, \ldots, L_d}(X) \sum_v [k(v):k]
				\left\{\max_{x \in \an{X}_v} \left| -\log \|1(x)\|_{\LL_{0,v}}\right| \right\}.
		\eenn
	Let $s$ be a rational section of $\LL_0$ that restricts to the section $1$ on the generic fiber. 
	Write $[\Div(s)] = \sum_v \sum_j  \ord_{W_{v,j}}(s) W_{v,j}$ and $[\XX_v] = 
	\sum_j m(v,j) W_{v,j}$  as in the remarks preceding this lemma. 
	The function $x \mapsto -\log \|1(x)\|_{\LL_{0,v}}$ must assume its maximum value at one
	of the $\xi_{v,j}$ \cite[2.4.4]{Berkovich_Spectral_Theory_1990}. 
	Using \eqref{Equation: Model Function value versus order function}, we have
		\benn
			\ba
				\left| \chern{\LL_0} \cdots \chern{\LL_d}\right| 
				&\leq \sum_v \sum_j |\ord_{W_{v,j}}(s)| 
					\left| \chern{\LL_1} \cdots \chern{\LL_d} \cdot [W_{v,j}]\right|  \\
				& =  \sum_v \sum_j \left| -\log \|1(\xi_{v,j})\|_{\LL_{0,v}} \right| m(v,j) \ 
					\chern{\LL_1} \cdots \chern{\LL_d} \cdot [W_{v,j}] \\
				&\leq \sum_v \left\{\max_{x \in \an{X}_v} \left| -\log \|1(x)\|_{\LL_{0,v}}\right| \right\}
				 \chern{\LL_1} \cdots \chern{\LL_d} \cdot [\XX_v].
			\ea
		\eenn
	In the second inequality we dropped the absolute values on the intersection with $W_{v,j}$ by 
	using the relative semipositivity of the line bundles $\LL_1, \ldots, \LL_d$. The result now follows
	immediately from Lemma~\ref{Lemma: Conservation of Number}.
\end{proof}


\subsubsection{Existence of the Intersection Pairing}

	Now we define the intersection pairing. We will proceed in several steps. Unless stated otherwise, 
	we let $X$ be a projective variety over $K$, and $\met{L}_0, \ldots, \met{L}_d$ will be adelic
	metrized line bundles on $X$. The idea is to take property (i) of 
	Theorem~\ref{Theorem: Intersection properties} as the definition when the metrics are induced
	by models, and then to use Lemma~\ref{Lemma: Metric bound on intersections} to control the 
	intersection number when passing to limits of model metrics. \\
	
	\noindent \textbf{Step 1} (Metrics induced by relatively semipositive models)
		Suppose $(\XX_i, \LL_i)$ is a $\BB$-model of $(X, L_i^{e_i})$ that induces the given metric
		on $L_i$, and suppose further that each $\LL_i$ is relatively semipositive. 
		Using the Simultaneous Model Lemma (Lemma~\ref{Lemma: Simultaneous Model 
		Lemma}), we obtain a single $\BB$-model $\XX$ and models of the $L_i^{e_i}$ that induce the 
		given metric on $L_i$. We abuse notation and denote these models on $\XX$ by $\LL_i$. 
		Then we define
			\benn
				\achern{\met{L}_0}\cdots \achern{\met{L}_d} = 
				\frac{\chern{\LL_0}\cdots \chern{\LL_d}}{e_0\cdots e_d}.
			\eenn
		To see that this is well-defined, it suffices to take $\XX'$ to be another $\BB$-model of $X$
		and $\LL_i'$ to be models of $L_i^{e_i'}$ that also induce the given metrics. (Note that the
		exponents $e_i'$ need not equal the $e_i$.) In order to prove that this data gives the same 
		intersection number, it is enough to prove that 
			\benn
				\chern{(\LL_0')^{e_0}} \cdots \chern{(\LL_d')^{e_d}} 
				= \chern{\LL_0^{e_0'}} \cdots \chern{\LL_d^{e_d'}}.
			\eenn
		
		By another application of the Simultaneous Model Lemma, we can
		find a single $\BB$-model $\YY$ of $X$, birational morphisms $\pr: \YY \to \XX$  and 
		$\pr': \YY \to \XX'$ that are isomorphisms on generic fibers over $\BB$, and models 
		$\MM_i = \pr^*\LL_i^{e_i'}$ and 
		$\MM_i' = \pr'^*(\LL_i')^{e_i}$ of $L_i^{e_i e_i'}$ that induce the given metrics on $L_i$. 
		By the projection formula, we are reduced to showing
			\be \label{Equation: Telescoping preparation}
				\chern{\MM_0}\cdots \chern{\MM_d} = \chern{\MM_0'}\cdots \chern{\MM_d'}. 
			\ee
		Observe that $\MM_i$ and $\MM_i'$ may be different line bundles on
		$\YY$, but they are models of the same line bundle on $X$ and they induce the same 
		metrics on it.
		
		 Proving \eqref{Equation: Telescoping preparation} uses a telescoping sum argument. 
		 To set it up, note that since the metrics on $L_i$ induced by $\MM_i$ and $\MM_i'$ agree, 
		 Lemma~\ref{Lemma: Metric bound on intersections} shows that
		 	\benn
				\chern{\MM_0} \cdots \chern{\MM_{i-1}} \chern{\MM_i \otimes (\MM_i')^{\vee}}
					\chern{\MM_{i+1}'} \cdots \chern{\MM_d'} = 0.
			\eenn
		Therefore
			\benn
				\ba
				| \chern{\MM_0}&\cdots \chern{\MM_d} - \chern{\MM_0'}\cdots \chern{\MM_d'}| \\
				&= \left|\sum_{i=0}^d \chern{\MM_0} \cdots \chern{\MM_{i-1}} \chern{\MM_i 
					\otimes (\MM_i')^{\vee}} \chern{\MM_{i+1}'} \cdots \chern{\MM_d'} \right| \\
				&\leq \sum_{i=0}^d |\chern{\MM_0} \cdots \chern{\MM_{i-1}} \chern{\MM_i 
					\otimes (\MM_i')^{\vee}} \chern{\MM_{i+1}'} \cdots \chern{\MM_d'}  | = 0.
				\ea
			\eenn
		This completes the first step. \\

		\noindent \textbf{Step 2} (Arbitrary semipositive metrized line bundles)
		Let $\met{L}_0, \ldots, \met{L}_d$ be semipositive metrized line bundles on $X$ with underlying
		algebraic bundles $L_0, \ldots, L_d$, respectively. For each place $v$, denote the metric on 
		$L_{i,v}$ by $\metric_{i,v}$. By definition, there exists a sequence of $\BB$-models 
		$(\XX_{i,m}, \LL_{i,m})$ of the pairs $(X, L_i^{e_{i,m}})$ such that
			\begin{itemize}
				\item $\LL_{i,m}$ is relatively semipositive on $\XX_{i,m}$ for every $i, m$;
				\item There exists an open set $\UU \subset \BB$ such that for place $v \in \UU$, 
					each index $i$ and all $m$, we have $\metric_{{\LL_{i,m}},v}^{1/e_{i,m}} = 
					\metric_{i,v}$. 
				\item For each place $v \not\in \UU$ and each $i$, the sequence of metrics
					$\metric_{{\LL_{i,m}},v}^{1/e_{i,m}}$ converges uniformly to $\metric_{i,v}$.
			\end{itemize}
		Define $\met{L}_{i,m}$ to be the semipositive metrized line bundle having algebraic 
		bundle $L_i$ and the metrics induced by $\LL_{i,m}$. Now we define the arithmetic intersection 
		number to be
			\be \label{Equation: Limit Intersection Definition}
				\achern{\met{L}_0} \cdots \achern{\met{L}_d} = \lim_{m_0, \ldots, m_d \to \infty}
				 \achern{\met{L}_{0,m_0}} \cdots \achern{\met{L}_{d,m_d}}.
			\ee
		Each of the intersection numbers on the right is well-defined by Step~1, so we have to show
		that the limit exists and that it is independent of the sequence of models.
		
		To prove that the limit in \eqref{Equation: Limit Intersection Definition} exists, we take
		$(m_0, \ldots, m_d)$ and $(m_0', \ldots, m_d')$ to be  two $(d+1)$-tuples of positive integers 
		and set $\LL_i = \LL_{i,m_i}^{e_{i,m_i'}}$, $\LL_i' = \LL_{i,m_i'}^{e_{i,m_i}}$, and 
		$e_i = e_{i,m_i}e_{i,m_i'}$. Then on the generic fibers, we have
			\benn
				\LL_i|_X = \LL_{i,m_i}^{e_{i,m_i'}}|_X = L_i^{e_i} \qquad
				\LL_i'|_X = \LL_{i,m_i'}^{e_{i,m_i}}|_X = L_i^{e_i}.
			\eenn
		Further, we may assume that for each $i$, $\LL_i$ and $\LL_i'$ are line bundles on a single 
		$\BB$-model $\XX$ (Simultaneous Model Lemma). Then our definition of the intersection
		pairing in Step~1 gives
			\be \label{Equation: Step 2 rearrangement}
				\ba
					 \achern{\met{L}_{0,m_0}} \cdots &\achern{\met{L}_{d,m_d}} - 
					  	\achern{\met{L}_{0,m_0'}} \cdots \achern{\met{L}_{d,m_d'}}  \\
					  &= \frac{\chern{\LL_{0,m_0}}\cdots \chern{\LL_{d,m_d}}}
					  	{e_{0,m_0}\cdots e_{d,m_d}} - 
						\frac{\chern{\LL_{0,m_0'}}\cdots \chern{\LL_{d,m_d'}}}
					  	{e_{0,m_0'}\cdots e_{d,m_d'}}\\
					  &= \frac{\chern{\LL_0}\cdots \chern{\LL_d} - 
					  	\chern{\LL_0'}\cdots \chern{\LL_d'}}{e_0 \cdots e_d}.
				\ea
			\ee
			
		Now fix $\varepsilon > 0$. For all places $v \in \UU$, we have 
				$\metric_{\LL_i,v}^{1/e_i} =  \metric_{\LL_i',v}^{1/e_i} = \metric_{i,v}$.
		For $v \not\in \UU$ and $m_i, m_i'$ sufficiently large, uniform convergence gives
			\benn
				\dist_v\left(\metric_{\LL_{i,v}}^{1/e_i}, \metric_{\LL'_{i,v}}^{1/e_i}\right) < \varepsilon.
			\eenn
		Just as in Step~1, we use a telescoping argument and apply 
		Lemma~\ref{Lemma: Metric bound on intersections}:
			\be \label{Equation: Step 2 Bound}
				\ba
					| \chern{\LL_0}&\cdots \chern{\LL_d} - \chern{\LL_0'}
						\cdots \chern{\LL_d'}| \\
					&= \left|\sum_{i=0}^d \chern{\LL_0} \cdots \chern{\LL_{i-1}} \chern{\LL_i 
						\otimes (\LL_i')^{\vee}} \chern{\LL_{i+1}'} \cdots \chern{\LL_d'} \right| \\
					&\leq \sum_{i=0}^d \left|\chern{\LL_0} \cdots \chern{\LL_{i-1}} \chern{\LL_i 
						\otimes (\LL_i')^{\vee}} \chern{\LL_{i+1}'} \cdots \chern{\LL_d'} \right| \\
					&\leq \sum_{i=0}^d \deg_{L_0^{e_0}, \ldots, \widehat{L_i^{e_i}}, 
						\ldots, L_d^{e_d}}(X) \sum_v [k(v):k] \dist_v\left(\metric_{\LL_i,v}, 
						\metric_{\LL'_i,v}\right) \\
					&< \varepsilon \prod_{i=0}^d e_i 
						\left(\sum_i \deg_{L_0, \ldots, \widehat{L_i}, 
						\ldots, L_d}(X) \right) \left(\sum_{v \not\in \UU} 
						[k(v):k] \right).
				\ea
			\ee
		Combining \eqref{Equation: Step 2 rearrangement} and \eqref{Equation: Step 2 Bound} shows
		that the sequence $\{\achern{\met{L}_{0,m_0}} \cdots \achern{\met{L}_{d,m_d}}\}$ is 
		Cauchy, which is tantamount to showing the the limit in 
		\eqref{Equation: Limit Intersection Definition} exists.
	
		To see that the limit in \eqref{Equation: Limit Intersection Definition} is independent of the
		sequence of models chosen, for each $i$ we let $\{(\XX_{i,m}', \LL_{i,m}')\}$ be another 
		sequence of models with associated metrics converging to the given ones on $L_i$.  Then we 
		obtain a third sequence 
			\benn
				\left\{(\XX_{i,m}'', \LL_{i,m}'')\right\} = \left\{(\XX_{i,1}, \LL_{i,1}), (\XX_{i,1}', 
				\LL_{i,1}'), (\XX_{i,2}, \LL_{i,2}), (\XX_{i,2}', \LL_{i,2}'), \ldots \right\}.
			\eenn
		This sequence also induces metrics converging uniformly to $\met{L}_i$, so by our above
		work we know that the limit in \eqref{Equation: Limit Intersection Definition} exists for this
		sequence. Therefore the limits over odd and even terms must agree, which is precisely what
		we wanted to prove.	
		
		Note that the symmetry and multilinearity of the pairing are guaranteed immediately by virtue
		of the same properties for the classical intersection pairing.\\

		\noindent \textbf{Step 3} (Integrable metrized line bundles)
		Now we extend the pairing by linearity since any integrable metrized line bundle 
		$\met{L}$ can be written $\met{L} = \met{L}' \otimes(\met{L}'')^{\vee}$ with $\met{L}'$ and
		$\met{L}''$ semipositive. As there may be multiple ways of decomposing 
		$\met{L}$ as a difference of semipositive metrized line bundles, a question of uniqueness
		arises. This is easily settled however using Lemma~\ref{Lemma: Metric bound on intersections}
		and we illustrate it only in the simplest case to avoid unnecessary notation.
		
		Let $\met{L}_1, \ldots, \met{L}_d$ be semipositive metrized line bundles and let $\met{L}$
		be an integrable metrized line bundle with a decomposition 
		$\met{L}=\met{L}' \otimes(\met{L}'')^{\vee}$ as above. Then
			\be \label{Eq: Integrable}
				\achern{\met{L}} \achern{\met{L}_1} \cdots \achern{\met{L}_d} = 
					\achern{\met{L}'} \achern{\met{L}_1} \cdots \achern{\met{L}_d} -
					\achern{\met{L}''} \achern{\met{L}_1} \cdots \achern{\met{L}_d}.
			\ee
		If $\met{L}$ has a second decomposition $\met{M}' \otimes(\met{M}'')^{\vee}$ as a difference of 
		semipositive metrized line bundles, it follows that $\met{L}' \otimes \met{M}'' = 
		\met{L}'' \otimes \met{M}'$. Each side of this last equality is semipositive with the same 
		underlying algebraic bundle and the same metrics, and so it follows from 
		Lemma~\ref{Lemma: Metric bound on intersections} that
			\benn
				\achern{\met{L}' \otimes \met{M}''} \achern{\met{L}_1} \cdots \achern{\met{L}_d} =
				\achern{\met{L}'' \otimes \met{M}'} \achern{\met{L}_1} \cdots \achern{\met{L}_d}.
			\eenn
		Splitting each side into two terms using linearity and rearranging shows that the expression
		in~\eqref{Eq: Integrable} is indeed well-defined.

	
\subsubsection{Proof of Theorem~\ref{Theorem: Intersection properties}}

	In the previous section we constructed an intersection pairing on $\overline{\Pic}(X)^{d+1}$, 
	and it is straight forward to check that the construction gives properties (i)-(iii) by using the 
	symmetry and multilinearity of the classical intersection product, 
	Lemma~\ref{Lemma: Metric bound on intersections}, and a limiting argument. Conversely, since we
	used (i) as our definition, and since the behavior under limits is 
	governed by (ii), our construction gives the only intersection product satisfying these two properties. 
	
	Therefore, it only remains to prove property (iv). Intuitively, the idea is that if $\YY$ and 
	$\XX$ are $\BB$-models of $Y$ and $X$, respectively, then the morphism
	$\varphi: Y \to X$ extends to a rational map $\widetilde{\varphi}: \YY \dashrightarrow  \XX$
	that is generically finite and has degree equal to the degree of $\varphi$.
	
	By linearity, it suffices to prove (iv) when $\met{L}_0, \ldots, \met{L}_d$ are semipositive metrized
	line bundles. Let us also assume for the moment that the metrics are induced by relatively 
	semipositive models $\LL_i$ of $L_i^{e_i}$. As per usual, we may
	use the Simultaneous Model Lemma to suppose all of these model line bundles live on a single 
	$\BB$-model $\XX$ of $X$. By Lemma~\ref{Lemma: Pullback of metrics} there exists a 
	$\BB$-model $\YY$ of $Y$ and a morphism $\lift{\varphi}: \YY \to \XX$ such that 
	$\lift{\varphi}_K = \varphi$. In this case, we find that $(\YY, \lift{\varphi}^*\LL_i)$ is a $\BB$-model
	of the pair $(Y, \varphi^*L_i^{e_i})$. By the projection formula for classical intersection products, 
	we get
		\benn	
			\ba
				\achern{\varphi^*\met{L}_0} \cdots \achern{\varphi^*\met{L}_d} &= 
					\frac{\chern{\lift{\varphi}^*\LL_0} \cdots \chern{\lift{\varphi}^*\LL_d}}
					{e_0 \cdots e_d} \\
				&= \deg(\lift{\varphi}) \ \frac{\chern{\LL_0} \cdots 
					\chern{\LL_d}}{e_0 \cdots e_d} \\
				&= \deg(\lift{\varphi}) \ \achern{\met{L}_0} \cdots \achern{\met{L}_d}.
			\ea		
		\eenn
	Passing to the generic fibers of $\YY$ and $\XX$ over $\BB$ does not change their function fields;
	hence, $\deg(\lift{\varphi}) = \deg(\varphi)$. This completes the proof when the metrics
	are all induced by models. 
	
	Since $\varphi: Y \to X$ is surjective, we know the induced morphism $\an{\varphi}: \an{Y}_v \to \an{X}_v$ 
	of analytic spaces is too \cite[Prop.~3.4.6]{Berkovich_Spectral_Theory_1990}.
	So if $\met{L}_i$ is an arbitrary semipositive metrized line bundle and $(\XX_i, \LL_i)$ is a model 
	of $(X, L_i^{e_i})$, we have
		\benn
			\dist_v\left(\varphi^*\metric_{\LL_i, v}^{1/e_i}, \varphi^*\metric_{\met{L}_i, v}\right)= 
			\dist_v\left(\metric_{\LL_i, v}^{1/e_i}, \metric_{\met{L}_i, v}\right),
		\eenn 
	where $\metric_{\met{L}_i, v}$ denotes the given metric on $L_i$ at the place $v$. This shows
	that if $\{\met{L}_{i,m}\}_m$ is a sequence of semipositive metrized line bundles induced from 
	models for which the metrics converge uniformly to those of $\met{L}_i$, then 
	$\{\varphi^*\met{L}_{i,m}\}_m$ is a sequence of semipositive metrized line bundles for which the
	metrics converge uniformly to those of $\varphi^*\met{L}_i$.
	This observation coupled with the above work
	in the model case gives the desired conclusion.


\subsection{Model Functions} 
\label{Section: Model Functions}

	In this section we fix a projective variety $X$ over $K$ and a place $v$ of $K$. As always, we 
	will identify $v$ with a closed point of the curve $\BB$. 

	A function $f: \an{X}_v \to \RR$ will be called a \textit{model function} if it is 
	of the form 
		\[
			f(x) = - \log \|1(x)\|^{1/e}
		\]
	for some formal metric $\|\cdot \|$ on $O_{\an{X}_v}$ and some positive integer $e$.
	Every such function is continuous on $\an{X}_v$. The set of model functions forms
	a $\QQ$-vector space. (The tensor product and inverse of formal metrics is again a formal metric.)
	The importance of model functions stems from the following result:
	
\begin{lem}[Gubler] \label{Lemma: Density of Model Functions}
	The space of model functions is uniformly dense in the space of real-valued 
	continuous functions on $\an{X}_v$.
\end{lem}

\begin{proof}
	This is the content of Theorem~7.12 of \cite{Gubler_Local_Heights_of_Subvarieties_1998}.
	One should note that, while the author assumes at the outset of section~7
	that the field $K$ is algebraically closed, he makes no use of it
	(nor is it needed for his references to the papers of Bosch and Lutkeb\"ohmert). 
\end{proof}

	As our intersection theory for adelic metrized line bundles is defined via intersection theory on
	$\BB$-models, we need to be able to relate model functions to global models in order to 
	perform computations with them. A result of Yuan provides this relation.

\begin{lem}[Yuan] \label{Lemma: Lifting Model Functions} 
	Suppose $\|\cdot\|_v$ is a formal metric on the trivial bundle over $\an{X}_v$, $e \geq 1$ is an integer, and
	$f(x) = -\log \|1(x)\|_v^{1/e}$ is a model function. Then there exists a $\BB$-model $(\XX, \OO(f))$ 
	of $(X,O_X^e)$ such that the metrics on $O_X$ determined by $\OO(f)$ are trivial at all places 
	$w \not= v$, and at the place $v$ the metric is $\|\cdot\|_v^{1/e}$. The line bundle $\OO(f)$ 
	admits a rational section $s$ such that $s|_X = 1$ and such that the support of the divisor of $s$ 
	lies entirely in the fiber over the point $v \in \BB$.
\end{lem}
	
\begin{proof}
	When $e=1$, this is precisely the content of the statement and proof
	of \cite[Lemma~3.5]{Xinyi_Arithmetic_Bigness_2006}. 
	The proof executed there in the case where
	$K$ is a number field applies equally well to the present situation.
	To extend to the case $e>1$, we start with a model function $f(x) = -\log \|1(x)\|_v^{1/e}$. Let
	$g = ef$. We may apply the case $e=1$ to the model function $g$ to get a 
	$\BB$-model $(\XX, \OO(g))$ of $(X, O_X)$ with trivial metrics at all places $w \not= v$ and the 
	metric $\|\cdot\|_v$ at $v$. Now view $\OO(g)$ as a model of the bundle $O_X^{e}$ via the 
	canonical isomorphism $O_X \simarrow O_X^e$ sending $1$ to $1^{\otimes e}$. This is 
	precisely the $\BB$-model we seek, as the metric at $v$ is now $\|\cdot\|_v^{1/e}$.
\end{proof}

	Momentarily we will prove a fundamental formula for intersecting the line bundle $\OO(f)$
	with the Zariski closure of a point of $x$ in some model $\XX$. First we need some notation.
	If $x$ is any closed point of $X$, note that $x$ breaks up into finitely many closed points over 
	$K_v$ --- one for each extension of the valuation $v$ to the residue field $K(x)$ at $x$ 
	(cf. Proposition~\ref{Proposition: fiber ring calculation}). 
	Let $\orb(x)$ be the image of this set of points under the canonical inclusion of 
	$|X_{K_v}| \hookrightarrow \an{X}_v$, where $|X_{K_v}|$ is the set of closed points of
	$X_{K_v}$. Another way to view $\orb(x)$ is as the image of $\an{\{x\}}_v$ in $\an{X}_v$. 
	Let $\deg_v(y) = [K_v(y):K_v]$ be the degree of the residue field of the closed point 
	$y \in X_{K_v}$ as an extension over $K_v$. These degrees satisfy the relation
	$\deg(x) = \sum_{y \in \orb(x)} \deg_v(y)$. (This is a classical fact proved in the appendix.  
	It can also be deduced from the discussion at the end of section~\ref{Section: Associated measure} by integrating 
	a nonzero constant function.)

\begin{lem} \label{Lemma: Intersection vs evaluation of f}
	Suppose $f$ is a model function on $\an{X}_v$ induced by
	a formal metric on $O_{\an{X}_v}$. Let $(\XX,\OO(f))$ be a $\BB$-model of $(X,O_X)$
	as in Yuan's lemma. If $x$ is a closed point of $X$, then
		\be \label{Equation: f against the Galois measure}
			\chern{\OO(f)}\cdot [\overline{x}] = [k(v):k]\sum_{y \in O_v(x)}\deg_v(y)f(y) ,
		\ee
	where $\overline{x}$ is the Zariski closure of $x$ in $\XX$ and $k(v)$ is the residue field
	of the point $v \in \BB $.
\end{lem}

\begin{proof}

	We begin the proof by interpreting the contribution of a point $y \in \orb(x)$ to the
	sum in~\eqref{Equation: f against the Galois measure} in
	terms of lengths of modules over a neighborhood in the formal completion of $\XX$ along the 
	closed fiber over the point $v$. Then we interpret the intersection number 
	$\chern{\OO(f)}\cdot \overline{x}$ in terms of the same quantities by working on the 
	formal completion of $\overline{x}$ along its closed fiber over $v$. 
	
	A point $y \in \orb(x)$ corresponds to a finite extension of fields $K_v(y) / K_v$ and a 
	$K_v$-morphism $\Spec K_v(y) \to X_{K_v}$. Let $R$ be the valuation ring of $K_v(y)$; 
	i.e., the integral closure of $K_v^{\circ}$ in $K_v(y)$. By properness, we obtain a lift to a 
	$K_v^{\circ}$-morphism $\tilde{y}: \Spec R \to \XX_{K_v^{\circ}}$. Now take an open affine 
	$W =\Spec A$ around the image of the closed point of $\Spec R$ via $\tilde{y}$ over which 
	$\OO(f) \otimes K_v^{\circ}$ is trivial. For topological reasons, $\tilde{y}$ factors through 
	$W$. Set $M =\left(\OO(f) \otimes K_v^{\circ}\right)|_W$; it is a free
	$A$-module of rank $1$. Then there exists a section $u \in M$ that generates 
	$M$ as an $A$-module and, by Yuan's lemma, a rational section $s$ such that 
	$s|_{W_{K_v}} = 1$. Write $s = (\gamma_1 / \gamma_2)u$ for some $\gamma_1, \gamma_2 \in A$. 
			
	Let $\mm_v$ and $\mm_R$ be the maximal ideals of $\Kvcirc$ and $R$, respectively. 
	Both of the latter are discrete valuation rings; we denote the corresponding normalized valuations by	
	$\ord_v$ and $\ord_R$, respectively. (Recall that this means a uniformizer has valuation
	$1$.) Extend these valuations to the fraction fields $K_v$ and $K_v(y)$ in the usual way.
	The normalized absolute value on $K_v$ is given by $|c|=\exp(- \ord_v(c))$ for any
	$c \in K_v$. If $e_y$ is the ramification index of $K_v(y)$ over $K_v$, then the absolute 
	value extends uniquely to $K_v(y)$, and is given by the formula $|c|=\exp(- \ord_R(c)/e_y)$.
	
	 Passing to the $\mm_v$-adic completion of everything in sight gives a morphism 
	 $\hat{y}: \Spf(R) \to \Spf(\hat{A})$, and we denote by 
	 $\hat{\gamma}_i$ the image of $\gamma_i$ in $\hat{A}$. 
	 Let $\alpha_{\hat{y}}: \hat{A} \to R$ be the induced morphism of complete $K_v^{\circ}$-algebras. 
	 By definition, we now have
	 	\benn
			\ba
				f(y) &= - \log \|1(y)\| = -\log\left| \frac{\hat{\gamma}_1}
				{\hat{\gamma}_2} (y) \right| =
				-\log\left|\frac{\alpha_{\hat{y}}(\hat{\gamma}_1)}
				{\alpha_{\hat{y}}(\hat{\gamma}_2)}\right| \\
				&= e_y^{-1} \ord_R\left( \alpha_{\hat{y}}(\hat{\gamma}_1) / 
				\alpha_{\hat{y}}(\hat{\gamma}_2) \right).
			\ea
		\eenn
	Letting $k_R = R / \mm_R$ and $k_v = \Kvcirc / \mm_v$ be the relevant residue fields, we have
	$[K_v(y):K_v] = e_y[k_R:k_v]$. This follows, for example, from the degree formula for 
	extensions of Dedekind rings. Therefore 
		\be \label{Equation: f contribution versus lengths}
			\deg_v(y)f(y) = [k_R:k_v]\ord_R\left( \alpha_{\hat{y}}(\hat{\gamma}_1) / 
				\alpha_{\hat{y}}(\hat{\gamma}_2) \right).
		\ee
	Note that $k_v$ is canonically isomorphic to the residue field of $\OO_{\BB,v}$ because the 
	completion of this local ring is precisely $\Kvcirc$; i.e., $k_v \cong k(v)$. 
			
	Now we turn to the intersection number $\chern{\OO(f)}\cdot [\overline{x}]$. 
	Let $\pi: \XX \to \BB$ be the structure morphism, and 
	let $j:\overline{x} \into \XX$ be the closed immersion of $\overline{x}$ with its reduced
	subscheme structure. There exists a rational section $s$ such that $\supp [\Div(s)]$ is 
	contained entirely in $\pi^{-1}(v)$, and $s|_X = 1$. This is the same section $s$ that was used
	above (prior to restriction and base change). As $\overline{x}$ is proper and quasi-finite
	over $\BB$, it is finite over $\BB$. To compute the intersection number 
	$\chern{\OO(f)}\cdot [\overline{x}]$ we may restrict to an affine neighborhood $U = \Spec C$ 
	of $v$. Let $\overline{x}_U = \Spec T$, where $T$ is a finite domain over $C$.  We also call $j$
	the morphism	$\Spec T \hookrightarrow \XX_U$. Let $N = j^*\left(\OO(f)|_{\XX_U}\right)$ 
	be the corresponding $T$-module, and we will also write $s$ for the image of our rational 
	section in $N$. By definition, we have
		\be \label{Equation: Affine intersection number}
			\chern{\OO(f)}\cdot [\overline{x}] 
			=  \sum_{\substack{t \in \Spec T \\ \text{closed points}}} [k(t):k] \ord_t(s).
		\ee
	Here $\ord_t(s) =
	l_{T_{\mm_t}}\left(T_{\mm_t}/ \sigma_1T_{\mm_t}\right) - l_{T_{\mm_t}}\left(T_{\mm_t}
	 / \sigma_2T_{\mm_t}\right)$, where $\mm_t$ is the maximal ideal of $T$ corresponding to
	 the point $t$, and $s$ corresponds to  $\sigma_1 / \sigma_2$ for 
	 some $\sigma_1, \sigma_2 \in T_{\mm_t}$
	 under an isomorphism $N  \otimes T_{\mm_t} \simarrow T_{\mm_t}$. It is
	 independent of the choice of isomorphism and of the choice of $\sigma_i$ (cf. 
	 \cite[Appendix A.3]{Fulton_Intersection_Theory_1998}).
	  
	 To say that $\supp [\Div(s)]$ lies in the fiber over $v$ means that $\ord_t(s) = 0$ whenever the closed point
	 $t$ does not lie over $v$. Thus localizing on the base $U$ in~\eqref{Equation: Affine intersection number} 
	 preserves all of the quantities in the sum, and so we may replace $T$ by $T \otimes_C \OO_{\BB,v}$. 
	 We continue to call this semi-local 
	 ring $T$. But length is preserved by flat residually trivial base extension, so we may even pass
	 to the $\mm_v$-adic completion without affecting the quantities 
	 in~\eqref{Equation: Affine intersection number}. Now $\hat{\OO}_{\BB,v} = K_v^{\circ}$, and 
	 $\hat{T} = T \otimes_{\OO_{\BB,v}} K_{v}^{\circ} = \prod_{i=1}^r \hat{T}_i$, where the 
	 maximal ideals of $T$ are $\mm_1, \ldots, \mm_r$ and $\hat{T}_i$ is the $\mm_v$-adic 
	 completion of $T_{\mm_i}$ (cf. \cite[Thm~8.15]{Matsumura_CRT_1989}). Note also that
	 the residue fields of $\OO_{\BB,v}$ and $K_v^{\circ}$ are canonically isomorphic.
	 Equation~\eqref{Equation: Affine intersection number} 
	 now becomes
	 	\be \label{Equation: Completed intersection number}
			\chern{\OO(f)}\cdot [\overline{x}] 
			= [k(v):k]\sum_{i=1}^r [k(\hat{t}_i):k(v)] \ord_{\hat{t}_i}(\hat{s}) ,
		\ee
	where $\hat{t}_1, \ldots, \hat{t}_r$ are the closed points of $\Spec \hat{T}$, and 
	$\hat{s}$ is the image of $s$ in the $\mm_v$-adic completion of $N$.
	
	By construction, base changing
	$\overline{x} \into \XX$ to $\Kvcirc$ gives a commutative diagram 
		\benn	\label{Diagram: Base change of closed immersion}
			\xymatrix@!0{
					 \Spec \hat{T} \ar@{=}[rr] &  
					& \overline{x}_{\Kvcirc} \ar[rr] \ar '[dr][ddrr] \ar[ddll] 
					& & \XX_{\Kvcirc} \ar[dd]  \ar[ddll]\\	 
					  & & & & & \\	
					  \overline{x} \ar[rr] \ar[ddrr] & & \XX \ar[dd] & & \Spec \Kvcirc \ar[ddll] \\
					  & & & & & \\
					  &  &  \BB & & 
			}
		\eenn
	The horizontal maps are closed immersions. By virtue of this diagram, we see that the points of
	$\orb(x)$ are in bijective correspondence with the 
	generic points of $\Spec \hat{T}$, which in turn are in bijective correspondence with the 
	generic points of the disjoint closed subschemes $\Spec \hat{T}_i$. 
	Moreover, given $y \in \orb(x)$ corresponding to the 
	generic point of $\Spec \hat{T}_i$, the $\Kvcirc$-morphism 
	$\tilde{y}:\Spec R \to \XX_{\Kvcirc}$ constructed
	above factors through the closed immersion $\Spec \hat{T}_i \into \XX_{\Kvcirc}$. This gives 
	the equality of residue fields $[k_R:k_v] = [k_R: k(\hat{t}_i)][k(\hat{t}_i):k(v)]$.
	
	Comparing~\eqref{Equation: f contribution versus lengths} and
	\eqref{Equation: Completed intersection number}, we see the proof will be complete 
	once we show that
		\benn
			\ord_{\hat{t}_i}(\hat{s})
			 = [k_R:k(\hat{t}_i)]\ord_R\left( \alpha_{\hat{y}}(\hat{\gamma}_1) / 
				\alpha_{\hat{y}}(\hat{\gamma}_2) \right).	
		\eenn
	Again for topological reasons, $\Spec R \to \Spec \hat{A}$ factors through $\Spec \hat{T}_i$, so 
	we find that the composition
	$\hat{A} \to \hat{T}_i \to R$ equals the homomorphism $\alpha_{\hat{y}}: \hat{A} \to R$ from
	before. Furthermore, the definitions are such that the elements $\hat{\sigma_j}$
	used to compute $\ord_{\hat{t}_i}(\hat{s})$ may be chosen to correspond to
	$\alpha_{\hat{y}}(\hat{\gamma_j})$ under the homomorphism $\hat{T}_i \to R$. As $R$ is
	the integral closure of $\hat{T}_i$ in $K_v(y)$, the desired equality is easily deduced from the following 
	well-known formula upon setting $S = \hat{T}_i$ and $a = \hat{\sigma}_i$.
\end{proof}

\begin{lem}[{\cite[Example A.3.1]{Fulton_Intersection_Theory_1998}}] 
	Let $S$ be a one-dimensional local noetherian domain. For any $a \in S$, we have 
		\[
			l_S(S/aS) = \sum_R l_R(R / aR)[R/\mm_R: S / \mm_S],
		\]
	where the sum is over all discrete valuation rings of the fraction field of $S$ that dominate $S$.
\end{lem}


\subsection{Associated Measures} \label{Section: Associated measure}

	As before, let $X$ be a projective variety of dimension $d$ over the function field $K$, and fix
	a place $v$ for the entirety of this section. For semipositive metrized line bundles 
	$\met{L}_1, \ldots, \met{L}_d$ on $X$, we will define a bounded Borel measure 
	$\chern{\met{L}_1}\cdots \chern{\met{L}_d}$ on $\an{X}_v$. In order to avoid extra notation, we
	do not indicate the dependence of the measure on the place $v$ as it will be apparent from context.
	Any model function $f: \an{X}_v \to \RR$ induces an 
	integrable metrized line bundle $\met{O_X(f)}$ on $X$, and the measure is defined by
		\benn
			\int_{\an{X}_v} f \ \chern{\met{L}_1}\cdots \chern{\met{L}_d} = 
			\achern{\met{O_X(f)}}\achern{\met{L}_1} \cdots \achern{\met{L}_d}.
		\eenn
	This approach through global intersection theory has the advantage of being technically
	easy to define. However, it obscures the fact (which we shall prove) that the measure depends 
	only on the metrics of the $\met{L}_i$ at the place $v$. One could also develop local intersection
	theory on formal schemes over $\Kvcirc$ and define the associated measures purely in terms of
	local intersection products. This is the viewpoint taken by Gubler; for a nice synopsis of the 
	properties of local intersection theory, see \cite[\S2]{Gubler_Bogomolov_2007}. These measures
	were originally defined by Chambert-Loir \cite{Chambert-Loir_Measures_2005} in the number
	field case using the formula of Theorem~\ref{Theorem: Measure Properties}(i) below, 
	and then by passing to the limit using 
	the local intersection theory of Gubler \cite{Gubler_Local_Heights_of_Subvarieties_1998}. 
	
	For $f$ a continuous function on $\an{X}_v$, define $\met{O_X(f)}$ to be 
	the adelic metrized line bundle with underlying bundle $O_X$, the trivial 
	metric at all places $w \not= v$ and the metric $\|1(x)\|_v = e^{-f(x)}$ at 
	$v$. 
	
\begin{lem} \label{Lemma: Continuous functions are integrable}
	If $f$ is a model function on $\an{X}_v$, the adelic metrized line bundle $\met{O_X(f)}$ is integrable.
\end{lem}	

\begin{proof}	
	If $f = -\log\metric_v^{1/e}$, then Yuan's Lemma 
	(Lemma~\ref{Lemma: Lifting Model Functions}) allows one to construct a $\BB$-model 
	$(\XX, \OO(f))$ of $(X, O_X^e)$ such that the associated adelic metrized line bundle is $\met{O_X(f)}$. 
	As $\XX$ is projective, we can write 
	$\OO(f) = \LL_1 \otimes \LL_2^{\vee}$ for some ample line bundles $\LL_1$, $\LL_2$ on $\XX$. 
	They are \textit{a fortiori} relatively semipositive, and so they induce semipositive metrized line
	bundles $\met{L}_1, \met{L}_2$ such that $\met{O_X(f)} = \met{L}_1 \otimes \met{L}_2^{\vee}$.
	Thus $\met{O_X(f)}$ is integrable.
\end{proof}

\begin{lem} \label{Lem: Bounded linear}
	Let $X$ be a projective variety of dimension~$d$ over the function field $K$ and $v$ a place of $K$. For any
	choice of semipositive metrized line bundles $\met{L}_1, \ldots, \met{L}_d$, the association
		\[
			f \mapsto \achern{\met{O_X(f)}}\achern{\met{L}_1} \cdots \achern{\met{L}_d}
		\]
	defines a bounded linear functional on the $\QQ$-vector space of model functions on $\an{X}_v$
	with the uniform norm.
\end{lem}

\begin{proof}
	It is apparent from the definition that if $f$ and $g$ are two model functions, then 
	$\met{O_X(f)} \otimes \met{O_X(g)} = \met{O_X(f+g)}$. Thus the map in question is
	additive. Next note that $\frac{1}{n}f$ is a model function for any $n \geq 1$ whenever
	$f$ is a model function (view $O_X$ as $O_X^n$ via the isomorphism $1 \mapsto 1^{\otimes n}$). 
	It is therefore an easy consequence of additivity that the map in question is $\QQ$-linear as
	desired. 
	
	The map is bounded by Theorem~\ref{Theorem: Intersection properties}(ii).
\end{proof}

	The space of model functions on $\an{X}_v$ is dense in the linear space 
	$\mathcal{C}(\an{X}_v, \RR)$ of real-valued continuous functions endowed with the uniform norm 
	(Lemma~\ref{Lemma: Density of Model Functions}), so the association in 
	the previous lemma extends to a bounded linear functional on $\mathcal{C}(\an{X}_v, \RR)$.
	By the Riesz representation theorem, we may identify it with a Borel measure on 
	$\an{X}_v$. Denote this measure by $\chern{\met{L}_1} \cdots \chern{\met{L}_d}$. Evidently
	we require $d= \dim X \geq 1$ for this notation to be sensible, an annoyance we will remedy
	at the end of this section.

\begin{thm} \label{Theorem: Measure Properties}
	Let $X$ be a projective variety of dimension~$d$ over the function field $K$ and $v$ a place of $K$. If 
	$\met{L}_1, \ldots, \met{L}_d$ denote semipositive metrized line bundles on $X$, then the 
	following properties hold for the measures $\chern{\met{L}_1}\cdots \chern{\met{L}_d}$:
	
	\begin{enumerate}
		
		\item Suppose $X$ is normal, that $\XX$ is a normal $\BB$-model of $X$, 
			and that $\LL_1, \ldots, \LL_d$ are models on $\XX$ of $L_1^{e_1}, \ldots, L_d^{e_d}$, 
			respectively, that induce the metrized line bundles $\met{L}_1, \ldots, \met{L}_d$. Let
			$[\XX_v] = \sum m(j) [W_j]$ with each $W_j$ irreducible, and let $\delta_{\xi_j}$ denote
			the Dirac measure at the unique point $\xi_j \in \an{X}_v$ that reduces to the generic point 
			of $W_j$.
			(See the remarks preceding Lemma~\ref{Lemma: Metric bound on intersections}.) Then
				\[
					\chern{\met{L}_1}\cdots \chern{\met{L}_d} =
					\sum_j m(j) \ \frac{\chern{\LL_1}\cdots \chern{\LL_d}\cdot [W_j]}
					{e_1 \cdots e_d} \delta_{\xi_j}.
				\]

		\item $\chern{\met{L}_1}\cdots \chern{\met{L}_d}$ is symmetric and multilinear in
			$\met{L}_1, \ldots, \met{L}_d$.
			
		\item $\chern{\met{L}_1}\cdots \chern{\met{L}_d}$ is a nonnegative measure.\footnote{
			Measure theory texts would called this a positive measure.} 
			
		\item If $\met{L}_1$ and $\met{L}_1'$ have the same underlying algebraic bundle and
			identical metrics at the place $v$, then 
			 	$\chern{\met{L}_1}\cdots \chern{\met{L}_d} =
				\chern{\met{L}_1'}\cdots \chern{\met{L}_d}$ as measures on $\an{X}_v$.
				
		\item The measure $\chern{\met{L}_1}\cdots \chern{\met{L}_d}$ has total mass
			$[k(v):k]\deg_{L_1, \ldots, L_d}(X)$.
			
		\item If $Y$ is another projective $K$-variety and $\varphi: Y \to X$ is a generically finite 
			surjective morphism, then
			\[
				\varphi_* \left\{ \chern{\varphi^*\met{L}_1}\cdots \chern{\varphi^*\met{L}_d}
				\right\} = \deg(\varphi) \chern{\met{L}_1}\cdots \chern{\met{L}_d}.
			\]

	\end{enumerate}
\end{thm}

\begin{proof}

	(i) It suffices to prove that both measures integrate the same way against a model function of the form
		$f = -\log \|1\|_v$.  Use Yuan's Lemma to choose a model
		$(\XX, \OO(f))$ of $(X, O_X)$ and a rational section $s$ of $\OO(f)$ such that the support of $[\Div(s)]$
		is contained in the fiber $\XX_v$. Then we have 
			\benn
				\ba
					\int_{\an{X}_v} f \ \chern{\met{L}_1}\cdots \chern{\met{L}_d} &=
					\frac{\chern{\OO(f)}\chern{\LL_1}\cdots \chern{\LL_d}}{e_1 \cdots e_d} \\
					&= \sum_j \ord_{W_j}(s) \ \frac{\chern{\LL_1}\cdots \chern{\LL_d} \cdot 
					[W_j]}{e_1 \cdots e_d}.
				\ea
			\eenn
		Applying \eqref{Equation: Model Function value versus order function}
		of section~\ref{Section: Intersection Numbers} shows us that
		$\ord_{W_j}(s) = m(j) f(\xi_j)$, which implies the result.
		
	(ii) This follows immediately from Theorem~\ref{Theorem: Intersection properties}(iii).
	
	(iii) It suffices to show that if $f=-\log \|1\|_v$ is a nonnegative model function, then  
		$\int_{\an{X}_v} f \chern{\met{L}_1} \cdots \chern{\met{L}_d} \geq 0$. We may also
		assume that all of our metrized line bundles are induced by models $\LL_1, \ldots, \LL_d$
		by using a limit argument. 
		Apply Yuan's lemma to get a line bundle $\OO(f)$ that induces the metrized line bundle
		$\met{O_X(f)}$, and let $s$ be a rational section of $\OO(f)$ whose associated divisor is
		supported in $\XX_v$. As $f \geq 0$, we deduce that $[\Div(s)]$ is effective (cf. 
		\eqref{Equation: Model Function value versus order function}). Then
			\[
				\int_{\an{X}_v} f \chern{\met{L}_1} \cdots \chern{\met{L}_d} =
				\chern{\LL_1}\cdots\chern{\LL_d} \cdot [\Div(s)] \geq 0 
			\]
		because the intersection of relatively semipositive line bundles on components of the
		fiber $\XX_v$ is nonnegative.

	(iv) It suffices to show that
		\[
			\int_{\an{X}_v} g \ \chern{\met{L}_1} \cdots \chern{\met{L}_d} = 
			\int_{\an{X}_v} g \ \chern{\met{L}'_1} \cdots \chern{\met{L}_d}	
		\]
	for any model function $g: \an{X}_v \to \RR$. In terms of intersection numbers, we must show
		\[
			\achern{\met{O_X(g)}} \achern{\met{L}_1} \cdots \achern{\met{L}_d} =
			\achern{\met{O_X(g)}} \achern{\met{L}'_1} \cdots \achern{\met{L}_d}.
		\]
	By linearity, this reduces to proving that if $\met{L}_1$ is an integrable metrized line bundle with
		underlying bundle $O_X$ and the trivial metric at $v$, and if $\met{L}_2, \ldots, \met{L}_d$ 
		are arbitrary semipositive metrized line bundles, then 
			\[
				\achern{\met{O_X(g)}} \achern{\met{L}_1} \cdots \achern{\met{L}_d} = 0.
			\]
			
		We know $\met{L}_1$ must have the trivial metric at almost all places, so there exist
		finitely many places $w_1, \ldots, w_n$ of $K$ and continuous functions $f_{w_i}: \an{X}_{w_i} \to \RR$
		such that 
			\[
				\met{L}_1 = \met{O_X(f_{w_1})} \otimes \cdots \otimes \met{O_X(f_{w_n})}.
			\]
		 We may
		assume that no $w_i = v$. Again by linearity, we may reduce to the case 
		$\met{L}_1 = \met{O_X(f_w)}$ for some continuous
		function $f_w$ with $w \not= v$. By a limit argument, we may further suppose that
		$f_w$ is a model function and that $\met{L}_2, \ldots, \met{L}_d$ are induced by 
		models $\LL_2, \ldots, \LL_d$ on some $\BB$-model $\XX$. Using Yuan's lemma
		(and the Simultaneous Model Lemma), we can find a line bundle $\OO(f_w)$ on $\XX$
		that induces $\met{O_X(f_w)}$ and a rational section $s$ of $\OO(f_w)$ with associated divisor supported
		entirely in the fiber $\XX_w$. Finally, 
		use Yuan's lemma again to get a line bundle $\OO(g)$ on $\XX$ (perhaps after replacing
		$\XX$ with a dominating model) equipped with a rational section $t$ whose divisor is supported in 
		$\XX_v$. Then
			\[
				\achern{\met{O_X(g)}} \achern{\met{L}_1} \cdots \achern{\met{L}_d} =
			 	\chern{\OO(g)} \chern{\OO(f_w)} \chern{\LL_2} \cdots \chern{\LL_d} = 0, 
			\]
		since the section $t$ is regular and invertible when restricted to $\XX_w$. Thus (iv) is proved.
					
	(v) Take any $\BB$-model $\XX$ of $X$. The cycle $[\XX_v]$ is a Cartier divisor, and we can 
	use it to define the constant model function $1$. Indeed, if $\pi$ is a uniformizer of $\OO_{\BB,v}$, 
	then $\pi$ is a local equation for $\XX_v$ on $\XX$. Consider the line bundle
	$\OO_{\XX}([\XX_v])$; it induces the metrized
	 line bundle $\met{O_X(f)}$, where $f(x) = -\log\|1(x)\|_v = -\log |\pi|_v = 1$. 
	 
	 By a limiting argument we may assume that the metrized line bundles
	 $\met{L}_1, \ldots, \met{L}_d$ are induced by models 
	 $\LL_1, \ldots, \LL_d$ of $L_1^{e_1}, \ldots, L_d^{e_d}$, respectively. Now
	 Lemma~\ref{Lemma: Conservation of Number} shows
		\[
			\int_{\an{X}_v} 1 \  \chern{\met{L}_1}\cdots \chern{\met{L}_d} =
				\frac{\chern{\LL_1}\cdots \chern{\LL_d}\cdot [\XX_v]}{e_1\cdots e_d}  =
				\frac{[k(v):k] \deg_{L_1^{e_1}, \ldots, L_d^{e_d}}(X)}{e_1\cdots e_d}.
		\]
				
	(vi) This is an easy consequence of Theorem~\ref{Theorem: Intersection properties}(iv)
		and the fact that $\met{O_Y(f \circ \varphi)} = \varphi^*\met{O_X(f)}$ for any continuous
		function $f$.
\end{proof}

	Parts (iii) and (v) of the above theorem indicate a natural normalization for these measures. For
	a semipositive metrized line bundle $\met{L}$ with ample underlying bundle $\met{L}$ and a place
	$v$ of $K$, define a probability measure by
		\[
			\mu_{\met{L},v} = \frac{\chern{\met{L}}^d}{[k(v):k] \deg_L(X)}.
		\]
	Given any subvariety $Y \subset X$, we can similarly define a probability measure supported on 
	$\an{Y}_v$ since $\met{L}|_Y$ is also semipositive. If $j:\an{Y}_v  \hookrightarrow \an{X}_v$ is
	the canonical inclusion, then we set
		\[
			\mu_{Y, \met{L}, v} = \frac{j_*\left\{\chern{\met{L}|_Y}^{\dim Y}\right\}}
			{[k(v):k] \deg_L(Y)}.
		\]
	When $Y=X$, we see immediately that $\mu_{Y, \met{L}, v} = \mu_{\met{L},v}$. 
	
	Finally, we want to define $\mu_{\{x\}, \met{L}, v}$ for a closed point $x \in X$. For any model
	function $f$ on $\an{X}_v$, define
		\benn
			\int_{\an{X}_v} f d\mu_{\{x\}, \met{L}, v} 
			= \frac{\achern{\met{O_X(f)}|_{\{x\}}}}{[k(v):k]\deg(x)}.			
		\eenn
 	The proofs of Lemma~\ref{Lem: Bounded linear} and Theorem~\ref{Theorem: Measure Properties}
	apply here to show that $\mu_{\{x\}, \met{L}, v}$ extends to a Borel probability measure
	on $\an{X}_v$. Evidently it is independent of the semipositive line bundle $\met{L}$, but we have 
	chosen to retain it in the notation to preserve symmetry with $\mu_{Y, \met{L}, v}$ when $Y$ is
	a higher dimensional subvariety. By Lemma~\ref{Lemma: Intersection vs evaluation of f}, we have
	the appealing formula
		\benn \label{Equation: muL for a point}
			\mu_{\{x\}, \met{L}, v} =  
			\frac{1}{\deg(x)} \sum_{y \in \orb(x)} \deg_v(y) \delta_y,
		\eenn
	where $\delta_y$ is the point measure supported at $y$.


\subsection{Global Height Functions}
\label{Section: Heights}

	In this section we define normalized height functions associated to semipositive metrized line
	bundles. One of the most useful properties of height functions 
	with regard to arithmetic intersection theory is the transformation law that they satisfy when one
	changes some of the metrics by a constant. For example, this property will allow us to define
	canonical height functions and invariant measures associated to a dynamical system. 
	
	Suppose $L$ is an ample line bundle on $X$ and $\met{L}$ is a semipositive metrized line bundle with
	underlying bundle $L$. Then $\met{L}|_Y$ is semipositive for any subvariety
	$Y \subset X$. We define the height of such a subvariety by
		\[
			h_{\met{L}}(Y) = \frac{\achern{\met{L}|_Y}^{\dim Y + 1}}{(\dim Y + 1) \deg_L(Y)}.
		\]

	Recall that $\met{O_X(b)}$ is defined to be the adelic metrized line bundle with underlying bundle 
	$O_X$, the trivial metric at all places $w \not= v$ and the metric $\|1(x)\|_v = e^{-b}$ at 
	$v$.

\begin{thm} \label{Thm: General Heights}
	Suppose $X$ is a projective variety over $K$, $L$ is an ample line bundle on $X$, and
	$\met{L}$ is any semipositive metrized line bundle with underlying bundle $L$. 
		\begin{enumerate}
			\item If $\met{L}'$ is another semipositive metrized line bundle with the
				same underlying algebraic bundle $L$, then there exists a positive constant $C$ 
				such that for any subvariety $Y$ of $X$,
					\[
						\left| h_{\met{L}}(Y) - h_{\met{L}'}(Y) \right| \leq C.
					\]
				In fact, we may take 
					\be \label{Eqn: Height bound}
						C = \sum_v [k(v):k] \dist_v(\metric_{\met{L},v}, \metric_{\met{L}',v}).
					\ee		
						
			\item Fix a real number $b$ and a place $v$ of $K$. Then the adelic metrized line bundle
				$\met{L} \otimes \met{O_X(b)}$ is semipositive, and for any subvariety $Y$, we have
					\[
						h_{\met{L} \otimes \met{O_X(b)}}(Y) = h_{\met{L}}(Y) + b[k(v):k].
					\]
			
			\item Given any closed point $x$ and rational section $s$ of $L$ such that
				$x \not\in \supp \left(\Div(s)\right)$, we have the following local decomposition:
					\benn
						h_{\met{L}}(x) = \frac{-1}{\deg(x)}
						\sum_{v \in \BB} [k(v):k] \sum_{y \in \orb(x)} \deg_v(y) \log \|s(y)\|_{\met{L},v}.
					\eenn
		\end{enumerate}
\end{thm}

\begin{proof}
	To prove (i), we use the telescoping sum trick from \eqref{Equation: Step 2 Bound} in the 
	previous section, and Theorem~\ref{Theorem: Intersection properties}(ii). Set $r = \dim Y$. Then 
		\benn
			\ba
				\left| \achern{\met{L}|_{Y}}^{r+1} - \achern{\met{L}'|_{Y}}^{r+1} \right| 
					&\leq \sum_{j=0}^r \left|\achern{\met{L}|_{Y}}^j 
						\achern{\left(\met{L} \otimes (\met{L}')^{\vee}\right)\big|_{Y}} 
						\achern{\met{L}'|_{Y}}^{r-j} \right| \\
					&\leq \sum_{j=0}^r \deg_{L}(Y) 
						\sum_v [k(v):k] \dist_v\left(\metric_{\met{L},v}, \metric_{\met{L}',v}\right) \\
					&= C (r+1) \deg_L(Y),
			\ea
		\eenn
	where $C$ is the constant in \eqref{Eqn: Height bound}. Dividing both sides by $(r+1)\deg_L(Y)$
	and using the definition of height gives the result.
			
	For (ii), we note that arithmetic intersection numbers are continuous with respect to
	change of metric (Theorem~\ref{Theorem: Intersection properties}(ii)). Therefore it suffices to
	assume that $\met{L}$ is induced by a relatively semipositive 
	$\BB$-model $(\XX, \LL)$ of $(X, L^e)$. Let us also assume that $b=\frac{m}{n} \in \QQ$. 
	As in the proof of Theorem~\ref{Theorem: Measure Properties}, we can construct a 
	$\BB$-model of $O_X$ that induces the metrized line bundle $\met{O_X(b)}$ by taking the line 
	bundle $\MM = \OO_{\XX}(m[\XX_v])$ associated to the Cartier divisor $m[\XX_v]$ on $\XX$. It is a
	$\BB$-model of the trivial bundle $O_X$, and we may view it as a $\BB$-model of $O_X^n$ via
	the isomorphism $O_X \simarrow O_X^n$ carrying $1$ to $1^{\otimes n}$. If $\pi$ is a local 
	equation for $\XX_v$ on $\XX$, then we see that the metric on $O_X$ at $v$ induced by $\MM$ is given by 
	$\|1(x)\|_{\MM, v} = |\pi^{m}|^{1/n}_v = e^{-m/n}$. The metrics at all of the other places are 
	evidently trivial. 
	
	The line bundle $\MM$ is relatively semipositive on $\XX$. Indeed, take any curve $C$ supported
	in the fiber over a point $w \in \BB$. If $w \not= v$, then $C$ and $[\XX_v]$ are disjoint and 
	$[\XX_v] \cdot C = 0$. If $w=v$, then we note that $[\XX_v] = \pi^*[v]$, where $\pi: \XX \to \BB$
	is the structure morphism. Let $D$ be a divisor on $\BB$ linearly equivalent to $[v]$ such that 
	$v \not\in \supp(D)$. Then $\pi^*D$ is a divisor with support disjoint from $C$, and so 
	$[\XX_v]\cdot C = \pi^*D \cdot C = 0$.
	
	Notice that, as an operator on codimension-two cycles, $\chern{\MM}^2 = 0$ by a linear equivalence 
	argument similar to the one at the end of the last paragraph. Hence, for any subvariety $Y \subset X$ of 
	dimension $r$, we have 
		\benn
			\ba
				\achern{\left(\met{L} \otimes \met{O_X(b)}\right)\big|_{Y}}^{r+1} 
				&= \sum_{j=0}^{r+1} \binom{r+1}{j} \achern{\met{L}|_{Y}}^j 
					\achern{\met{O_X(b)}|_Y }^{r+1-j} \\
				&= \sum_{j=0}^{r+1} \binom{r+1}{j} \frac{\chern{\LL}^j \chern{\MM}^{r+1-j} 
					\cdot [\alg{Y}]}{e^j n^{r+1-j}} \\
				&= \frac{\chern{\LL}^{r+1}\cdot [\overline{Y}]}{e^{r+1}} + (r+1)\frac{\chern{\LL}^r
					\chern{\MM}\cdot[\overline{Y}]}{e^r n} \\
				&= \achern{\met{L}|_{Y}}^{r+1} + \frac{m}{n} (r+1) 
					\frac{ \chern{\LL}^r \cdot [\alg{Y}] \cdot [\XX_v]}{e^r} \\
				&=  \achern{\met{L}|_{Y}}^{r+1} + \frac{m}{n} (r+1) 
				\frac{ \chern{\LL|_{\overline{Y}}}^r \cdot [\left(\alg{Y}\right)_v]}{e^r } \\
				&= \achern{\met{L}|_{Y}}^{r+1} + \frac{m}{n}(r+1)[k(v):k]\deg_L(Y).
			\ea
		\eenn
	The last equality follows from Lemma~\ref{Lemma: Conservation of Number}. Applying the 
	definition of height immediately gives the result in the case $b = \frac{m}{n}$. The general case
	follows by continuity of arithmetic intersection numbers when we take a limit over rational 
	approximations of $b$.
	
	The proof of (iii) is similar to the proof of Lemma~\ref{Lemma: Intersection vs evaluation of f},
		so we omit it. 
\end{proof}

\section{Algebraic Dynamical Systems} \label{Section: Dynamics}

	In this section we review the facts necessary to work with algebraic dynamical systems defined
	over a function field, including the construction of the invariant metrics on the polarization
	of a dynamical system, the theory of (canonical) dynamical heights, and the invariant 
	measures for the dynamical system.

\subsection{Invariant Metrics} \label{Section: Invariant metrics}

	Here we are concerned with the existence and uniqueness properties of invariant metrics on the 
	polarization of an algebraic dynamical system. This will give us a natural
	semipositive metrized line bundle with which to define heights related to a dynamical system.
	
	Let $(X,\varphi,L)$ be an algebraic dynamical system over $K$ as in the introduction. Suppose 
	$\theta: \varphi^*L \simarrow L^q$ is an isomorphism with $q > 1$. For a place $v$ of $K$, 
	choose any initial metric $\metric_{1,v}$ on $L_v$. For example, it could be the metric induced
	by a $\BB$-model of $L$. We can construct an \textit{invariant metric} on 
	$L_v$ by Tate's limit process: by induction, define 
		\[
			\metric_{n+1,v} = \left(\varphi^*\metric_{n,v} \circ \theta^{-1}\right)^{1/q}.
		\] 
	Here $\varphi^*\metric_{n,v}$ denotes the metric on $\varphi^*L_v$ induced by pullback.
	It is well-known (cf. \cite[\S9.5]{Bombieri-Gubler_2006} or 
	\cite[\S2]{Zhang_Small_Points_1995}) that this sequence of metrics converges 
	uniformly to a continuous metric $\metric_{0,v}$ on $L_v$ with the following properties:
		\begin{enumerate}
			\item  The pullback by $\varphi$ agrees with the $q$th tensor power (up to the 
				isomorphism $\theta$):
				\[
					\metric_{0,v}^{\otimes q} \circ \theta = \varphi^*\metric_{0,v}.
				\]
			
			\item If $\theta$ is replaced by $\theta' = a \theta$ for some $a \in K^{\times}$, then 
				the corresponding metric constructed by Tate's limit process satisfies
				\[
					\|\cdot\|_{0,v}' = |a|_v^{\frac{1}{q-1}}\metric_{0,v}.
				\]
		\end{enumerate}
	Property (i) uniquely determines the metric. In the literature this metric is sometimes 
	called the ``canonical metric'' or an ``admissible metric.'' We adhere to the term 
	\textit{invariant metric} because it is only canonical up to a choice of 
	isomorphism $\theta$ by property (ii), and we feel the term ``admissible'' is already 
	overused in nonarchimedean geometry. We can interpret property (i) by saying that
	the family $\{\metric_{0,v}\}_v$ of invariant metrics provide the unique adelic metric 
	structure on $L$ such that the isomorphism $\theta: \varphi^*L \simarrow L^q$ becomes an isometry. 
	
	The above discussion settles the existence of invariant metrics at each place $v \in \BB$, 
	but we still need to show that they fit together to give a semipositive adelic metrized line bundle:
		\begin{itemize}
			\item[(iii)] There exists a sequence of $\BB$-models $(\XX_n, \LL_n)$ of $(X, L^{e_n})$ 
				such that each $\LL_n$ is nef, $\metric_{\LL_n,v}^{1/e_n} = \metric_{0,v}$ for 
				almost all $v$,  and $\metric_{\LL_n, v}^{1/e_n} \to \metric_{0,v}$ 
				uniformly for every other place $v$. In particular, the metrized line bundle 
				$\met{L}$ with underlying bundle $L$ and the family of metrics $\{\metric_{0,v}\}_v$ 
				is semipositive.
		\end{itemize}
	The first step in this direction is to construct a sequence of $\BB$-models that determine 
	metrics on $L$ according to Tate's limit process. 
	
	\begin{lem} \label{Lem: Nef Model}
		Let $X$ be a projective variety over $K$ and $L$ an ample line bundle on $X$. Then there exists
		a positive integer $e$ and a $\BB$-model $(\XX, \LL)$ of $(X, L^e)$ such that $\LL$ is nef.
	\end{lem}
	
	\begin{proof}
		This proof was adapted from a remark in the Notation and Conventions section of
	\cite[\S2.1]{Xinyi_Arithmetic_Bigness_2006}. Choose $e$ so that $L^e$ is very ample. Let 
			\[
				X \hookrightarrow \PP^N_K \hookrightarrow \PP^N_{\BB} =\PP^N \times \BB
			\] 
	be an embedding induced by $L^{e}$ followed by identifying $\PP^N_K$ with the 
	generic fiber of $\PP^N_{\BB}$, and set $\XX$ to be the Zariski closure of $X$ in 
	$\PP^N_{\BB}$ with the reduced structure. Let $\pi: \XX \to \BB$ be the restriction of the 
	second projection. Choose a collection of basepoint free global sections $s_0, \ldots, s_N$ of $L^e$, and
	let $\lift{s_i}$ be the section $s_i$ viewed as a rational section of $\OO_{\XX}(1)$. Let $D$ be an 
	ample Cartier divisor $\BB$ such that $[\pi^*D] + [\Div(\lift{s_i})]$ is effective for all $i$. Finally, 
	define $\LL = \OO_{\XX}(1) \otimes \pi^*\OO_{\BB}(D)$.
	
	We claim that $\LL$ is nef. Indeed, suppose $\YY$ is an irreducible curve on $\XX$. If $\YY$ is vertical --- 
	i.e., $\pi(\YY) = \{v\}$ for some closed point $v \in \BB$ --- then 
		\[
			\chern{\LL} \cdot [\YY] = \deg (\OO_{\XX}(1) |_{\YY}) > 0,
		\]
	since $\OO_{\XX}(1)$ is relatively ample. If $\YY$ is horizontal, then choose one of the sections $s_i$ of 
	$L^e$ such that $\YY_K \not\in \supp [\Div(s_i)]$. Then $\YY$ is not contained in 
	the support of $[\Div(\lift{s_i})]$, and it intersects properly with any subvariety of a vertical fiber. Hence,
		\benn
				\chern{\LL} \cdot [\YY] = [\Div(\lift{s_i})]_h \cdot [\YY]
					+ \left([\Div(\lift{s_i})]_f + \chern{\pi^*D}\right) \cdot [\YY] \geq 0,
		\eenn
	where $[\Div(\lift{s_i})] = [\Div(\lift{s_i})]_h + [\Div(\lift{s_i})]_f$ is the decomposition of this cycle into 
	its horizontal and vertical parts. 
	\end{proof}

	Returning to our construction, choose an initial $\BB$-model $(\XX_1, \LL_1)$ of $(X, L^{e_1})$ such 
	that $\LL_1$ is nef on $\XX_1$. The above lemma guarantees the existence of such a $\BB$-model. Define the 
	metric on $L_v$ to be $\metric_{1,v} =  \metric_{\LL_1,v}^{1/e_1}$. 
	
	We proceed induction. Suppose $(\XX_n, \LL_n)$ is a $\BB$-model of $(X,L^{e_n})$, 
	where $e_n = q^{n-1}e_1$ and $\LL_n$ is nef. Let $j_n:X \hookrightarrow \XX_n$ be the inclusion of the 
	generic fiber (always implicitly precomposed with the preferred isomorphism 
	$\iota_n: X \simarrow (\XX_n)_K$). Let $\Gamma_\varphi: X \to X \times_K X$ be the graph morphism. 
	Consider the commutative diagram
		\benn \label{Diagram: Graph Square}
			\xymatrix{
				X \ar@/^1pc/ [drr]^{j_n \circ \varphi} \ar[dr]^{\widetilde{\Gamma}_\varphi
			 	} \ar@/_1pc/[ddr]_{j_n} & & \\
				&  \XX_n \times_{\BB} \XX_n \ar[r]^{\ \ \pr_2} \ar[d]^{\pr_1} & \XX_n \ar[d] \\
				& \XX_n \ar[r] & {\BB} \\
			}
		\eenn
	where $\widetilde{\Gamma}_\varphi = (j_n \times j_n) \circ\Gamma_\varphi$. 
	Define $\XX_{n+1}$ to be 
	the Zariski closure of $\widetilde{\Gamma}_\varphi(X)$ in $\XX_n \times_{\BB} \XX_n$ with the 
	reduced subscheme structure. Take $\LL_{n+1} = \pr^*_2 \LL_n|_{\XX_{n+1}}$. The graph 
	morphism $\Gamma_\varphi$  and the tensor power 
	$\theta^{\otimes e_n}$ give the preferred isomorphisms 
	between $X$ and the generic fiber of $\XX_{n+1}$ and between $\varphi^*L^{e_n}$ and $L^{qe_n}$, 
	respectively. As always, we will make these identifications without comment in what follows.  
	Set $e_{n+1}=qe_n$, and define a metric on $L_v$ by $\metric_{n+1,v} = 
	\metric_{\LL_{n+1},v}^{1/e_{n+1}}$.
	Observe that $\LL_{n+1}$ is nef since it is the pullback of a nef line bundle. 

	The metrics on $L_v$ are, by construction, exactly as given by 
	Tate's limit process. This follows from the fact that formation of formal metrics 
	commutes with formal pullback (Lemma~\ref{Lemma: Pullback of metrics}), and a small computation:
		\benn
			\ba
				\metric_{n+1,v} = \metric_{\LL_{n+1},v}^{1/{e_{n+1}}} &= 
				\left(\metric_{\pr_2^*\LL_{n,v}|_{\XX_{n+1}}}^{1/e_n}\right)^{1/q} \\
				&= \left(\varphi^*\metric_{\LL_n,v}^{1/e_n} \circ \theta^{-1}\right)^{1/q} \\
				&= \left(\varphi^*\metric_{n,v} \circ \theta^{-1}\right)^{1/q}.
			\ea
		\eenn
		
	Moreover, we now show that almost all of the metrics constructed are stable under this
	pullback procedure. As $X$ is of finite type over $K$, there exists an open subset 
	$\UU \subset \BB$ such that the endomorphism $\varphi$ extends to a $\UU$-morphism  
	$\varphi_\UU:(\XX_1)_\UU \to (\XX_1)_\UU$, 
	and the isomorphism $\theta: \varphi^*L \simarrow L^q$ extends to an isomorphism 
	$\theta_\UU: \varphi_\UU^*\LL_1|_{\pi^{-1}(\UU)} \simarrow \LL_1^q|_{\pi^{-1}(\UU)}$. 
	The graph morphism $\widetilde{\Gamma}_\varphi$ extends over $\UU$ to give a closed immersion
	$(\XX_1)_\UU \to (\XX_1)_\UU \times_\UU (\XX_1)_\UU$. Consequently, its (scheme-theoretic) 
	image is exactly $(\XX_{2})_\UU$, so that $\XX_1$ and $\XX_2$ are isomorphic when restricted
	over $\UU$. Pulling back $\LL_2$ via this isomorphism and applying $\theta_\UU$ shows 
	$\Gamma_{\varphi_\UU}^*\LL_2 = \varphi_\UU^*\LL_1 \cong \LL_1^q$ over $\UU$. 
	As $\LL_2$ is a model of 
	$L^{e_2}$ via the graph morphism and the isomorphism $\theta$, we conclude that for each 
	place $v$ corresponding to a closed point of $\UU$, we have
		\[
			\metric_{2,v} = \metric_{\LL_2}^{1/e_2} = \metric_{\LL_1^q}^{1/qe_1}
			= \metric_{\LL_1}^{1/e_1} = \metric_{1,v}.
		\]
	
	The isomorphism between $(\XX_1)_\UU$ and $(\XX_2)_\UU$ allows us to extend the work in 
	the previous paragraph by induction to conclude that for each place $v$ of $\UU$, the metrics 
	$\metric_{n,v}$ on $L_v$ are equal for all $n$.


\subsection{Dynamical Heights} \label{Section: Dynamical Heights}
	
	Let the data $(X, \varphi, L)$, $\theta: \varphi^*L \simarrow L^q$, and $\met{L}$ be as in the 
	previous section.
	For a subvariety $Y \subset X$, we can define its dynamical height with respect to the dynamical
	system $(X, \varphi, L)$ by the formula
		\benn
			h_{\varphi}(Y) = h_{\met{L}}(Y) = 
			\frac{\achern{\met{L}|_Y}^{\dim Y + 1}}{(\dim Y + 1) \deg_L(Y)}.
		\eenn	

\begin{thm} \label{Theorem: Height properties}
	Let $(X, \varphi, L)$ be a dynamical system over $K$, $\theta: \varphi^*L \simarrow L^q$ an
	isomorphism, and $\met{L}$ the line bundle $L$ equipped with the corresponding invariant
	metrics $\{\metric_{0,v}\}_v$.
		\begin{enumerate}
			
			\item The height $h_\varphi$ is independent of the choice of isomorphism $\theta$.

			\item  For any subvariety $Y \subset X$, $h_{\varphi}(Y) \geq 0$.
			
			\item  For any subvariety $Y \subset X$, $h_{\varphi}(\varphi(Y)) = q h_{\varphi}(Y)$.

			\item If $Y$ is preperiodic for the map $\varphi$ (i.e., the forward orbit
			$\{\varphi^n(Y): n=1, 2, \ldots\}$ is finite), then $h_{\varphi}(Y) = 0$.		
		\end{enumerate}
\end{thm}

	Before turning to the proof, we will need the following

\begin{lem} \label{Lem: Finiteness}
	Let $(X, \varphi, L)$ be a dynamical system defined over $K$, and let $\varphi^*L \cong L^q$ for some
	integer $q > 1$. Then for any subvariety $Y \subset X$, the induced morphism $Y \to \varphi(Y)$ is finite
	of degree $q^{\dim Y}$.
\end{lem}

\begin{proof}
	Let $\psi: Y \to \varphi(Y)$ be the morphism induced by $\varphi$. First note that $\psi^*(L|_{\varphi(Y)})$ is
	ample on $Y$ since the restriction of an ample bundle to a subvariety is still ample, and
		\[
			\psi^*\left(L|_{\varphi(Y)}\right) = \left(\varphi^*L\right)|_Y \cong L^q|_Y = \left(L|_Y\right)^q.
		\]
	If $\psi(Z) = \{p\}$ for some subvariety $Z \subset Y$ and some point $p$, then 
	$\psi^*\left(L|_{\varphi(Z)}\right) \cong O_Z$, which can only be ample if $Z$ is reduced to a point. Hence
	$\psi: Y \to \varphi(Y)$ has finite fibers. As $X$ is projective, we see $\psi$ is a projective quasi-finite morphism, 
	and so it must be finite.
	
	If $r = \dim Y$, the projection formula gives
		\benn
				\chern{L}^r \cdot [Y] = q^{-r} \chern{\varphi^*L}^r \cdot[Y] 
				= q^{-r} \chern{L}^r \cdot \varphi_*([Y])  = \frac{\deg(\psi)}{q^r} \chern{L}^r \cdot [Y].
		\eenn
	As $L$ is ample, we may divide by $\chern{L}^r \cdot [Y]$ to conclude $\deg(\psi) = q^r$.
\end{proof}

\begin{proof}[Proof of Theorem~\ref{Theorem: Height properties}]

	(i) Let $\theta' = a\theta$ for some $a \in K^{\times}$. If $\met{M}$ is the metrized line bundle
	with underlying bundle $O_X$ and metric at the place $v$ given by 
	$\|1(x)\|_v = |a|_v^{\frac{1}{q-1}}$, 
	then the invariant metrized line bundle corresponding to $\theta'$ is 
	$\met{L}' = \met{L} \otimes \met{M}$ (property (ii) in 
	section~\ref{Section: Invariant metrics}). Take any model $\XX$ of $X$ and consider the
	line bundle $\OO(\Div(a))$ associated to the principal divisor $\Div(a)$ on $\XX$. We can view it
	as a model of $O_X^{q-1}$ on the generic fiber via isomorphisms 
	$O_X(\Div(a)) \cong O_X \cong O_X^{q-1}$. It is easy to check that the metric on $O_X$ at $v$
	given by $\OO(\Div(a))$  coincides with that of $\met{M}$. Letting $r = \dim Y$, we find that
		\benn
			\ba
		 		\achern{\met{L}'|_Y}^{r + 1} - \achern{\met{L}|_Y}^{r +1} &=
					\achern{\met{L}|_Y \otimes \met{M}|_Y}^{r+1} - \achern{\met{L}|_Y}^{r +1} \\
					&= \sum_{i=0}^{r} \binom{r + 1}{i} \achern{\met{L}|_Y}^i 
					\achern{\met{M}|_Y}^{r + 1-i}.
			\ea
		\eenn
	All of the terms in this sum involve an intersection with $\achern{\met{M}}$, and
	if we compute this intersection on a model we are forced to intersect with the principal divisor
	$\Div(a)$. Thus each term in the sum vanishes. 

	(ii) Arithmetic intersection numbers are continuous with respect to change of metric, so it 
		suffices to prove $\chern{\LL}^{\dim Y+1}\cdot [\overline{Y}] \geq 0$, whenever $(\XX, \LL)$
		is a $\BB$-model of $(X, L^e)$, $\LL$ is nef and $\overline{Y}$ is the Zariski closure of $Y$ 
		in $\XX$. Kleiman's theorem on intersections with nef divisors implies the desired inequality 
		\cite[Thm.~1.4.9]{Lazarsfeld_Positivity_2004}.

	(iii) Let $r = \dim Y$. By Lemma~\ref{Lem: Finiteness}, the morphism $\varphi$ restricts to a finite 
	morphism $Y \to \varphi(Y)$ of degree $q^r$. Theorem~\ref{Theorem: Intersection properties}(iv) implies
		\[
			h_{\varphi}(\varphi(Y)) = \frac{\achern{\met{L}|_{\varphi(Y)}}^{r+1}}{(r+1)
				\deg_L(\varphi(Y))}
			= \frac{\achern{\varphi^* \met{L}|_Y}^{r+1}}{(r+1)
				\deg_{\varphi^*L}(Y)}
			= q\frac{\achern{\met{L}|_Y}^{r+1}}{(r+1)
				\deg_{L}(Y)}
			 = q h_\varphi (Y).
		\]
		
	(iv) If the set $\{\varphi^n(Y): n=1, 2, \ldots\}$ is finite, then $\varphi^n(Y) = \varphi^m(Y)$ for 
		some $m > n \geq 1$. By the previous part, we have 
		
			\[
				q^n h_{\varphi}(Y) = h_\varphi (\varphi^n(Y)) 
				=h_\varphi (\varphi^m(Y)) = q^m h_\varphi (Y).
			\]
		As $q > 1$, we are forced to conclude that $h_{\varphi}(Y) = 0$.
\end{proof}
	
	As a special case of part (iv) of the previous theorem, we note that
		\benn \label{Eqn: Height vanishes}
			h_{\varphi}(X) = \frac{\achern{\met{L}}^{d+1}}{(d+1) \deg_L(X)} = 0.
		\eenn

\subsection{The Invariant Measure $\muinv$}
	\label{Section: Invariant measure}

	Choose an isomorphism $\theta: \varphi^*L \simarrow L^q$, and let $\met{L}$ be the line bundle 
	$L$ equipped with the invariant metrics constructed as above. Fix a place $v$ of $K$. 
	Define a Borel probability measure on $\an{X}_v$ by the formula
		\[
			\muinv = \mu_{\met{L},v} =\frac{\chern{\met{L}}^d}{[k(v):k] \deg_L(X)}.
		\]
	Here $\mu_{\met{L},v}$ and $\chern{\met{L}}^d$ are the measures constructed in 
	section~\ref{Section: Associated measure}. An argument similar to the one that proved
	Theorem~\ref{Theorem: Height properties}(i) shows that $\muinv$ is independent of the choice 
	of isomorphism $\theta$. Although it is not logically necessary for what follows, we give some 
	further commentary on these measures.

	Since $\varphi$ is finite of degree $q^d$ (Lemma~\ref{Lem: Finiteness}), we see that the measure $\muinv$ has 
	the following invariance property:
		\[
			\varphi_* \muinv = \muinv.
		\]	
	Indeed, it follows immediately from Theorem~\ref{Theorem: Measure Properties} and the fact that 
	$\varphi^*\met{L}$ is isometric to $\met{L}^q$. 
	
	Given any subvariety $Y \subset X$, we can also define the measure 
	$\mu_{Y, \varphi, v} = \mu_{Y, \met{L},v}$ as in section~\ref{Section: Associated measure}. 
	Lemma~\ref{Theorem: Measure Properties}(vi) can be used to show
		\[
			\varphi_*\mu_{Y, \varphi, v} = \mu_{\varphi(Y), \varphi, v}.
		\]
	
	An important example is the case when $X$ is a smooth geometrically connected projective 
	variety over $K$ and $v$ is a place of \textit{good reduction} for $(X, \varphi, L)$; i.e., there
	exists an open subvariety $\UU \subset \BB$ containing the point $v$, 
	a smooth $\UU$-model $(\XX, \LL)$ of $(X, L)$, a $\UU$-morphism 
	$\varphi_{\UU}: \XX \to \XX$ whose restriction to the generic fiber is precisely $\varphi$, and 
	an isomorphism $\varphi_{\UU}^*\LL \simarrow \LL^q$. Roughly, the dynamical system can be 
	reduced$\pmod v$. One can see from Theorem~\ref{Theorem: Measure Properties}(i) and our 
	description of Tate's limit process that there exists a point $\zeta \in \an{X}_v$ 
	such that $\muinv = \delta_\zeta$. The point $\zeta$ is the unique point mapping to the generic
	point of the special fiber $\XX_v$ under the reduction map $\an{X}_v \to \XX_v$. 
	Moreover, the forward invariance of the measure $\muinv$ implies that $\zeta$ is a fixed point 
	of the analytification of $\varphi$: $\an{\varphi}_v(\zeta) = \zeta$. 
		
	As a final remark, we mention a backward invariance property the measure $\muinv$ presumably
	possesses based on the work of Chambert-Loir \cite[\S2.8]{Chambert-Loir_Measures_2005} and
	others, although we do not provide any proof in the present article. There is a way to define a 
	trace map $\varphi_*$ on the space of continuous functions on $\an{X}_v$, and by duality a 
	pullback measure $\varphi^*\muinv$. It should then be true that $\varphi^*\muinv = q^d \muinv$. 
	As a consequence of this backward invariance property, if $(X, \varphi, L)$ has 
	good reduction at a place $v$, then $(\an{\varphi}_v)^{-1}(\zeta) = \zeta$, where 
	$\muinv = \delta_\zeta$ as in the previous paragraph. That is, $\zeta$ is a totally invariant point for 
	the morphism $\an{\varphi}_v$.  By analogy with the case of complex dynamical systems, we expect 
	that the invariant measure $\muinv$ can be completely characterized as the unique Borel probability
	 measure on $\an{X}_v$ such that
		\begin{itemize}
			\item $\varphi^*\muinv = q^d \muinv$, and
			\item $\muinv$ does not charge any proper subvariety of $X$: $\muinv(\an{Y}_v) = 0$
				for any proper subvariety $Y \subset X$.
		\end{itemize}
	See the articles of Chambert-Loir \cite{Chambert-Loir_Measures_2005} and Chambert-Loir / Thuillier \cite{Chambert-Loir_Thuillier_Mahler_Integrals_2008} for proofs that the above properties hold for the measure $\muinv$. It is not yet known if they determine the measure. See the article of Briend and Duval \cite{Briend_Duval_2001} for a discussion of such a characterization in the setting of complex dynamics.

\section{Proof of the Equidistribution Theorem} 
\label{Section: Proof of equidistribution}

	Our goal for this section is to prove Theorem~\ref{Theorem: Generic Equidistribution}. We will 
	deduce it from a stronger result that is more flexible for applications and also gives equidistribution
	of small subvarieties. As always, we let
	$X$ be a variety over the function field $K$. Let $\met{L}$ be a semipositive metrized line bundle
	on $X$ with ample underlying bundle $L$ satisfying the following two conditions:
		\begin{itemize}
			\item[(S1)] There exists a sequence of $\BB$-models $(\XX_n, \LL_n)$ of $(X, L^{e_n})$ 
				such that each $\LL_n$ is nef, $\metric_{\LL_n,v}^{1/e_n} = \metric_{0,v}$ for 
				almost all $v$,  and $\metric_{\LL_n, v}^{1/e_n} \to \metric_{0,v}$ 
				uniformly for every other place $v$.
			\item[(S2)] The height of $X$ is zero: $h_{\met{L}}(X) = 0$.
		\end{itemize}
		
	A \textit{net of subvarieties of $X$} consists of an infinite directed set $A$ and a 
	subvariety $Y_\alpha \subset X$ for each $\alpha \in A$. A net of subvarieties $(Y_\alpha)_{\alpha \in A}$ is
	called \textit{generic} if for any proper closed subset $V \subset X$, there exists $\alpha_0 \in A$ so that 
	$Y_\alpha \not\subset V$ whenever $\alpha \geq \alpha_0$. Equivalently, there does not exist a cofinal subset
	$A' \subset A$ such that $Y_\alpha \subset V$ for all $\alpha \in A'$. The net is called 
	\textit{small} if $\lim_{\alpha \in A} h_{\met{L}}(Y_\alpha) = 0$.
	
\begin{thm} \label{Thm: Stronger equidistribution}
	Let $X$ be a projective variety over the function field $K$ equipped with a semipositive metrized
	line bundle $\met{L}$ with ample underlying bundle $L$ satisfying conditions (S1) and (S2). 
	Let $(Y_\alpha)_{\alpha \in A}$ be a generic small net of subvarieties of $X$. 
	Then for any place $v$ of $K$, and for any continuous function $f: \an{X_v} \to \RR$,
    	we have
      		\benn
         		\lim_{\alpha \in A} \int_{\an{X}_v} f d\mu_{Y_\alpha, \met{L}, v} =
         		\int_{\an{X_v}} f d\mu_{\met{L}, v}.
      		\eenn
	That is, the net of measures $\left(\mu_{Y_\alpha, \met{L}, v} \right)_{\alpha \in A}$ converges 
	weakly to $\mu_{\met{L}, v}$. 
\end{thm}

	Before turning to the proof, let us indicate why Theorem~\ref{Theorem: Generic Equidistribution} 
	follows from Theorem~\ref{Thm: Stronger equidistribution}. Let $(X, \varphi, L)$ be a dynamical 
	system defined over the function field $K$. Choose an isomorphism
	$\theta: \varphi^*L \simarrow L^q$. Let $\met{L}$ be the semipositive metrized line bundle with 
	underlying bundle $L$ and the associated invariant metrics at all places as defined in
	section~\ref{Section: Invariant metrics}. Then property (iii) of the same section is precisely the
	condition (S1). As $h_\varphi = h_{\met{L}}$ (by definition), the discussion at the end
	of section~\ref{Section: Dynamical Heights} shows condition (S2).
	Thus the hypotheses of Theorem~\ref{Thm: Stronger equidistribution} on $\met{L}$ are satisfied. 
	Upon unraveling all of the definitions, the conclusion of Theorem~\ref{Theorem: Generic 
	Equidistribution} follows immediately from that of the above theorem. 
	
	In order to see why the above theorem is more useful in practice, consider a dynamical system 
	$(X, \varphi, L)$, and let $Y$ be any subvariety of $X$ such that $h_{\varphi}(Y) = 0$. If 
	$\varphi(Y) \not= Y$, then $Y$ cannot be considered as a dynamical system on its own. 
	Nevertheless, we find that $\met{L}|_Y$ is a semipositive metrized line bundle satisfying conditions
	(S1) and (S2), and so we can use the above theorem to deduce equidistribution 
	statements for generic small nets of subvarieties of $Y$. 

\begin{proof}[Proof of Theorem~\ref{Thm: Stronger equidistribution}]
	
	Fix a place $v$ of $K$. By Lemma~\ref{Lemma: Density of Model Functions} and a 
	limiting argument, it suffices to prove the theorem when $f = -\log \|1\|_v^{1/n}$ is a model function.
	By linearity of the integral, we may take $n=1$.
	Lemma~\ref{Lemma: Lifting Model Functions} allows us to assume that $f$ is induced by a 
	$\BB$-model $(\XX, \OO(f))$ of $(X,O_X)$. We also choose ample line bundles $\MM_1$
	and $\MM_2$ on $\XX$ so that $\OO(f) = \MM_1 \otimes \MM_2^{\vee}$. 
	Let $\met{M}_i$ be the metrized line bundle on $X$ determined by $\MM_i$.
	Finally, we assume that $\int_{\an{X}_v} f \ d \mu_{\met{L},v} > 0$ for the moment and remove this 
	hypothesis at the end of the proof.
	
	For any $N \geq 1$, we define  $\met{L}^N(f) := \met{L}^N \otimes \met{O_X(f)}$. 
	We wish to compute the degree of this metrized line bundle in two ways.
	For the first, we have
		\be \label{Equation: Variation Computation}
      			\ba
     				\achern{\met{L}^{ N}(f)}^{d+1} &= \left(N\achern{\met{L}} +
     				\achern{\met{O_X(f)}}\right)^{d+1} \\
     				&= N^{d+1}\achern{\met{L}}^{d+1} + N^d(d+1)\deg_L(X)[k(v):k]\int_{\an{X}_v} f 
				\ d\mu_{\met{L},v} + O\left(N^{d-1}\right) \\
				&= N^d(d+1)\deg_L(X)[k(v):k]\int_{\an{X}_v} f \ d\mu_{\met{L},v} + 
				O\left(N^{d-1}\right).
		      	\ea
   		\ee
	The integral appears by definition of the measure $\mu_{\met{L},v}$. The term 
	$\achern{\met{L}}^{d+1}$ vanishes because it is the numerator of $h_{\met{L}}(X)$
	(condition (S2)). The constant in the error term depends on $\met{L}$ and $f$. 

		On the other hand, we see that
   		\be \label{Equation: Siu Computation}
      			\ba
    				\achern{\met{L}^{ N}(f)}^{d+1} &= \achern{\met{L}^{ N} \otimes
     				\met{M}_1 \otimes \met{M}_2^{\vee}}^{d+1} \\
     				&= \sum_{i=0}^{d+1} \binom{d+1}{i}(-1)^{d+1-i}\ \achern{\met{L}^{ N} 
				\otimes \met{M}_1}^i \achern{\met{M}_2}^{d+1-i} \\
     				&= \achern{\met{L}^{ N} \otimes \met{M}_1}^{d+1} - (d+1) 
				\achern{\met{L}^{ N} \otimes
     				\met{M}_1}^d \achern{\met{M}_2} + O\left(N^{d-1}\right).
      			\ea
   		\ee
	Recall that we assumed $\int_{\an{X}_v} f  d\mu_{\met{L},v} > 0$. Comparing 
	\eqref{Equation: Variation Computation} and \eqref{Equation: Siu Computation} shows that
	for $N$ sufficiently large, 
		\benn
		 	\achern{\met{L}^{ N} \otimes \met{M}_1}^{d+1} - (d+1)\achern{\met{L}^{ N} \otimes
     				\met{M}_1}^d \achern{\met{M}_2} >0.
		\eenn
	We may fix such an $N$ for the remainder of the argument, 
	and as it will have no effect on the proof, we will replace $\met{L}^{ N}$ by $\met{L}$.
		
	Choose $\varepsilon > 0$. By condition~(S1) we may select a $\BB$-model $(\XX', \LL)$ of $(X, L^e)$ 
	such that $\LL$ is nef, the metrics on the associated adelic metrized line bundle $\met{L}'$ with underlying 
	bundle $L$ are equal to those of $\met{L}$ at almost all places, and the sum of the weighted 
	distances $[k(v):k] \dist_v(\metric_{\met{L},v}, \metric_{\met{L}',v})$ at the other places is bounded 
	by $\varepsilon$.  By the Simultaneous Model Lemma, we may 
	assume that $\XX' = \XX$ so that $\OO(f)$ and $\LL$ are line bundles on $\XX$. Furthermore, 
	continuity of intersection numbers with respect to changes in the metric allows us to assume that
		\benn 
			\chern{\LL \otimes \MM_1^e}^{d+1} 
				- (d+1)\chern{\LL \otimes \MM_1^e}^d \chern{\MM_2^e} > 0. 
		\eenn
		
	The necessary tool from algebraic geometry needed to move forward at this point is
	
\begin{siu}[{\cite[Theorem~2.2.15]{Lazarsfeld_Positivity_2004}}] \label{Theorem: Siu}
	Let $\mathcal{Y}$ be a projective variety of dimension $n$ 
	over the field $k$ and suppose $\mathcal{N}_1$ and $\mathcal{N}_2$ are 
	nef line bundles on $\mathcal{Y}$. 
	If  
		\[
			\chern{\mathcal{N}_1}^n -n\ \chern{\mathcal{N}_1}^{n-1}\chern{\mathcal{N}_2} > 0,
		\]
	then $\left(\mathcal{N}_1 \otimes \mathcal{N}_2^{\vee}\right)^{ r}$ has nonzero 
	global sections for $r \gg 0$.
\end{siu}

	We are in a position to apply Siu's theorem with 
	$\YY = \XX$, $n=d+1$, $\mathcal{N}_1 = \LL \otimes \MM_1^e$ and 
	$\mathcal{N}_2 = \MM_2^{ e}$. It follows that the line bundle $\left(\LL
	\otimes \MM_1^{ e} \otimes \MM_2^{ (-e)}\right)^{ r} = 
	\left(\LL \otimes \OO(f)^{ e}\right)^{ r}$ admits global sections for all $r \gg 0$.
	Fix such an $r$ and a nonzero global section $s$. As 
	$(Y_\alpha)_{\alpha \in A}$ is a generic net in $X$, there exists $\alpha_0$ such that 
	$\overline{Y_\alpha}$ does not lie in the support of $\Div(s)$ for any $\alpha \geq \alpha_0$. 
	This means $\chern{\left(\LL \otimes \OO(f)^e\right)^r}\cdot [\overline{Y_\alpha}]$ is an 
	effective cycle. As $\LL$ is nef, Kleiman's theorem \cite[Thm.~1.4.9]{Lazarsfeld_Positivity_2004} shows
		\benn
			\chern{\LL}^{\dim Y_{\alpha}}
				\chern{\left(\LL \otimes \OO(f)^e\right)^r}\cdot [\overline{Y_\alpha}]
				\geq 0, 
		\eenn
	or equivalently 
		\be \label{Equation: Positivity}
			\achern{\met{L}'|_{Y_{\alpha}}}^{\dim Y_{\alpha}} 
			\achern{(\met{L}' \otimes \met{O_X(f)})|_{Y_{\alpha}}} \geq 0.
		\ee
	Our precision in picking the metrics on $\met{L}'$ and  
	Theorem~\ref{Theorem: Intersection properties}(ii) show that
		\be \label{Equation: Discrepancy}
			\ba \Bigg| \achern{\met{L}'|_{Y_{\alpha}}}^{\dim Y_{\alpha}} 
			\achern{(\met{L}' \otimes \met{O_X(f)})|_{Y_{\alpha}}}
			- \ &\achern{\met{L}|_{Y_{\alpha}}}^{\dim Y_{\alpha}} 
			\achern{(\met{L} \otimes \met{O_X(f)})|_{Y_{\alpha}}} \Bigg| \\
			&\leq \varepsilon (\dim Y_{\alpha} + 1) \deg_L(Y_\alpha).
			\ea
		\ee
	From \eqref{Equation: Positivity} and \eqref{Equation: Discrepancy} we now get
		\benn
			\ba
				h_{\met{L}}(Y_\alpha) + [k(v):k] &(\dim Y_\alpha + 1)^{-1} 
					\int_{\an{X}_v} f \ d\mu_{Y_\alpha, \met{L}, v} \\ 
				&= \frac{\achern{\met{L}|_{Y_{\alpha}}}^{\dim Y_\alpha} 
				\achern{(\met{L} \otimes \met{O_X(f)})|_{Y_{\alpha}}}}
				{(\dim Y_\alpha + 1)\deg_L(Y_\alpha)} 
				\geq - \varepsilon
			\ea
		\eenn
	for all $\alpha \geq \alpha_0$. Taking the limit over $\alpha \in A$ in this last expression 
	and recalling $h_{\met{L}}(Y_{\alpha}) \to 0$ 
	proves that 
		\[
			\liminf_{\alpha \in A} \int_{\an{X}_v} f \ d\mu_{Y_\alpha, \met{L}, v}  \geq 
			-\frac{\varepsilon(d+1)}{[k(v):k]}.
		\]
	Finally, $\varepsilon$ is independent of $f$, so we conclude that
		\be \label{Equation: Galois measure inequality}
			\liminf_{\alpha \in A} \int_{\an{X}_v} f \ d\mu_{Y_\alpha, \met{L}, v}  \geq 0.
		\ee	
		
	This last inequality holds for any model function $f$ such that $\int_{\an{X_v}} f
	d\mu_{\met{L},v} > 0$. In order to lift this restriction, we take an arbitrary model
	function $f$ and consider the function $f_1 = f - \rho$, where $\rho \in \log
	\sqrt{|K_v^{\times}|} = \QQ$ is such that $\int_{\an{X_v}} f_1 d\mu_{\met{L},v} > 0$. 
	Constant functions of this form are model functions, 
	and so $nf_1$ satisfies all of the necessary hypotheses to make the above argument go 
	through for some positive integer $n$. (We need $nf_1$ to be the model function associated to 
	a formal metric --- not just the root of a formal metric.) Applying~\eqref{Equation: Galois 
	measure inequality} to $nf_1$ shows that
    		\benn
         		\liminf_{\alpha \in A} \int_{\an{X}_v} f \ d\mu_{Y_\alpha, \met{L}, v}  \geq \rho.
    		\eenn
	Letting $\rho \to \int_{\an{X_v}} f d\mu_{\met{L},v}$ from below preserves the
	positivity of the integral of $f_1$ and shows
    		\benn
         		\liminf_{\alpha \in A} \int_{\an{X}_v} f \ d\mu_{Y_\alpha, \met{L}, v}  
		\geq \int_{\an{X_v}} f d\mu_{\met{L},v}.
    		\eenn
	Finally, we may replace $f$ with $-f$ in this argument to obtain the
	opposite inequality. The proof is now complete.
\end{proof}


\section{Corollaries of the Equidistribution Theorem} \label{Section: Applications}

        
        

	Our first corollary of the equidistribution theorem shows that for a dynamical system 
	$(X, \varphi, L)$, the invariant measures $\muinv$ reflect the $v$-adic distribution of
	the preperiodic points of the morphism $\varphi$. Recall that a closed point $x \in X$ is called
	\textit{preperiodic} if its (topological) forward orbit $\{\varphi^n(x): n=1, 2, \ldots\}$ is a finite set.

\begin{cor}
	Let $(X, \varphi, L)$ be an algebraic dynamical system over the function field $K$. For any
	generic net of preperiodic closed points $(x_{\alpha})_{\alpha \in A}$ in $X$ and any place $v$, 
	we have the following weak convergence of measures on $\an{X}_v$:
		\[
			\lim_{\alpha \in A} \frac{1}{\deg(x_{\alpha})} \sum_{y \in \orb(x_\alpha)} \deg_v(y) \delta_y
			= \muinv.
		\]
\end{cor}

\begin{proof}
	This is immediate from Theorem~\ref{Theorem: Generic Equidistribution} upon noting that
	preperiodic points have dynamical height zero (Theorem~\ref{Theorem: Height properties}(iv)). 
\end{proof}

	The preceding corollary is meaningless unless we can find generic nets of preperiodic points. However, it
	is not difficult to show that preperiodic points in $X(\alg{K})$ are Zariski dense in $X$. Once Zariski
	density is established, it is not hard to construct a generic net of preperiodic points by a diagonalization
	argument; for example, see the argument at the beginning of the proof of
	Corollary~\ref{Corollary: too much support}.

	If $E$ is a finite extension of $K$, we let $[E:K]_{\operatorname{s}}$ be the separable 
	degree of $E$ over $K$. Write $|X|$ for the set of closed points of a variety $X$.

\begin{cor} \label{Corollary: too much support}
	Let $(X, \varphi, L)$ be an algebraic dynamical system defined over the function field $K$, let
	$Y$ be any subvariety of $X$, and let $n$ be a positive integer. Suppose there exists a place $v$ 
	of $K$ such that the support of the probability measure $\mu_{Y, \varphi, v}$ on $\an{X}_v$ 
	contains at least $n+1$ points. Then there exists a positive number $\varepsilon$ such that the set
		\[
			Y_n(\varepsilon):=\{y \in |Y|: h_{\varphi}(y) \leq \varepsilon \text{ and } 
			[K(y):K]_{\operatorname{s}} \leq n \}
		\]õ
	is not Zariski dense in $Y$. 
\end{cor}
	
\begin{proof}
	If the theorem fails, then $Y_n(\varepsilon)$ is Zariski dense for each $\varepsilon > 0$. 
	We begin by constructing a generic small net. Let $A$ be the collection of all ordered
	pairs $(F,\varepsilon)$ consisting of a proper Zariski closed subset $F$ of $Y$ and a positive real number 
	$\varepsilon$. Then $A$ becomes a directed set when we endow it with the partial ordering
		\benn
			(F, \varepsilon) \leq (F', \varepsilon') \Longleftrightarrow F \subseteq F' 
			\text{ and } \varepsilon \geq \varepsilon'.
		\eenn
	For each pair $(F, \varepsilon) \in A$, select a point 
	$y_{F,\varepsilon} \in Y_n(\varepsilon) \cap (Y \smallsetminus F)$, a feat that is possible because
	$Y_n(\varepsilon)$ is Zariski dense. One checks easily that the net of points $(y_{F,\varepsilon})$ is 
	generic and $h_\varphi(y_{F,\varepsilon}) \to 0$. For ease of notation, we now relabel this net as
	$(y_\alpha)_{\alpha \in A}$.
	
	Let $p_0, \ldots, p_n$ be distinct points of $\an{Y}_v$ in the support of $\mu_{Y, \varphi, v}$.
	By topological normality of analytic spaces associated to proper varieties 
	\cite[Thm.~3.5.3]{Berkovich_Spectral_Theory_1990}, we can choose an open 
	neighborhood $U_i$ of $p_i$ for each $i$ with pairwise disjoint closures.  
	Fix an index $i_0$. Inside $U_{i_0}$, choose a compact neighborhood 
	$W$ of $p_{i_0}$. By Urysohn's lemma 
	we may find a continuous function $f: \an{X}_v \to [0,1]$ such that $f|_W \equiv 1$ and
	$f|_{\an{X}_v \smallsetminus U_{i_0}} \equiv 0$. Then Theorem~\ref{Thm: Stronger equidistribution}
	shows 
		\benn
		 	\lim_{\alpha \in A} \frac{1}{\deg(y_\alpha)} \sum_{z \in \orb(y_\alpha)} \deg_v(z) f(z) 
			 =\int_{\an{X}_v} f d\mu_{Y, \varphi, v} \geq\mu_{Y, \varphi, v}(W) > 0.
		\eenn
	Hence there exists $\alpha_0 \in A$ such that  $\orb(y_\alpha) \cap U_{i_0} \not= \emptyset$ for 
	all $\alpha \geq \alpha_0$. Repeating this argument for each index $i$, we can find $\alpha_1 \in A$ so that for 
	any $i = 0 , \ldots, n$ and $\alpha \geq \alpha_1$, we have $\orb(y_\alpha) \cap U_i \not= \emptyset$.
		
	For each point $y_\alpha$, the set $O_v(y_\alpha)$ consists of at most $n$ points
	by Corollary~\ref{Corollary: Degree formula} in the appendix. But the $n+1$ sets $U_i$ are disjoint 
	by construction, so we have a contradiction.
\end{proof}

	When $h_\varphi (Y) > 0$, the last corollary can be proved using the Theorem of Successive Minima
	with $\varepsilon = h_\varphi(Y) / 2$. We recall the statement of the theorem and indicate how
	this works. Let $Y$ be a variety defined over the function field $K$, and let $\met{L}$ be a 
	semipositive metrized line bundle on $Y$ with ample underlying bundle $L$. Define the quantity
		\[
			e_1(Y, \met{L}) = \sup_{\substack{V \subset Y \\ \codim(V, Y) = 1}} 
				\left\{ \inf_{y \in |Y\smallsetminus V|}  h_{\met{L}}(y) \right\}, 
		\]
	where the supremum is over all closed subsets $V$ of $Y$ of pure codimension $1$, and the 
	infimum is over closed points of $Y \smallsetminus V$. The Theorem of Successive Minima tells us that
		\[
			e_1(Y, \met{L}) \geq h_{\met{L}}(Y).
		\]
	This inequality was originally discovered by Zhang when $K$ is replaced by a number field
	\cite[Thm.~5.2]{Zhang_Positive_Arithmetic_Varieties_1995}, and it was proved by 
	Gubler when $K$ is a function field \cite[Lem.~4.1]{Gubler_Bogomolov_2007}.  
		
	Now let $\met{L}$ be the semipositive metrized line bundle associated to a dynamical system 
	$(X, \varphi, L)$ and an isomorphism $\theta: \varphi^*L \simarrow L^q$. By the Theorem of 
	Successive Minima, given any $\delta > 0$ there exists a closed 
	codimension-$1$ subset $V \subset Y$ so that
		\[
			\inf \{h_\varphi (y): y \in |Y \smallsetminus V|\} > h_\varphi (Y) - \delta. 
		\]
	If $h_\varphi (Y) > 0$, then
	we may take $\delta = h_\varphi (Y)/2$. The corollary follows immediately with 
	$\varepsilon = h_\varphi (Y)/2$ since $Y_n(\varepsilon) \subset V$. In fact, this shows that 
	$\cup_{n \geq 1} Y_n(\varepsilon)$ is not Zariski dense in $Y$ when $h_\varphi(Y) > 0$.

\begin{cor} \label{Cor: Preperiodic points not dense}
	Let $(X, \varphi, L)$ be an algebraic dynamical system over the function field $K$, let $Y$
	be any subvariety, and let $n$ be a positive integer. Suppose there exists a place $v$ of $K$ such 
	that the support of the probability measure $\mu_{Y, \varphi, v}$ on $\an{X}_v$ contains at least $n+1$ points. 
	Then the set of preperiodic closed points contained in $Y$ of separable degree at most $n$ over 
	$K$ is not Zariski dense in $Y$. 
\end{cor}

	The problem with these last two results is that one must have some knowledge of the support of the
	 measure $\mu_{Y, \varphi, v}$ in order to utilize them. As we indicated at the end of 
	section~\ref{Section: Invariant measure}, the support of the measure $\muinv$ is precisely one
	point if $X$ is smooth and the dynamical system $(X, \varphi, L)$ has good reduction at the place 
	$v$. So we cannot apply the corollaries in the case of good reduction. 
	
	We expect a converse to be true. Suppose that $X$ is geometrically connected and smooth
	over $K$ (e.g., the projective space $\PP^d_K$). If $E$ is a finite extension of $K$ and $v$ 
	is a place of 
	$K$, we say that the dynamical system $(X, \varphi, L)$ has \textit{potential good reduction at $v$}
	 if there exists a place $w$ of $E$ lying over $v$ so that the 
	base-changed dynamical system $(X_E, \varphi_E, L \otimes E)$ has good reduction at $w$. If
	$(X, \varphi, L)$ does not have potential good reduction at $v$, then it has \textit{genuinely bad
	reduction at $v$}. 	With these definitions in mind, we present the following folk conjecture which, 
	when combined with Corollaries~\ref{Corollary: too much support} 
	and~\ref{Cor: Preperiodic points not dense}, would yield very pleasing arithmetic results:
		
\begin{conjecture}
	Let $(X, \varphi, L)$ be a dynamical system defined over the function field $K$, and suppose $X$
	is smooth and geometrically connected. Then the support of the measure $\muinv$ is either
	a single point or else Zariski dense corresponding to the cases where $(X, \varphi, L)$ has potential
	good reduction or genuinely bad reduction.
\end{conjecture}

	The conjecture is true when $X$ is a curve. See for example the manuscript of Baker and Rumely 
	\cite[\S10.4]{Baker_Rumely_Potential_Berkovich_Line_2008} for the case $X = \PP^1_K$. (Compare 
	the article \cite{Baker_Function_Fields_2006} of Baker for a similar statement and arithmetic 
	consequence.) In \cite{Gubler_Bogomolov_2007} Gubler's work shows that if $X$ is an abelian variety 
	with totally degenerate reduction at a place $v$, then $\muinv$ has Zariski dense support. For an elliptic 
	curve, totally degenerate reduction is the same as genuinely bad reduction. In fact, in this case there 
	is a topological subspace of $\an{X}_v$ homeomorphic to a circle in such a way that $\muinv$ is a 
	Haar measure on this circle.

\section{Appendix}
\label{Section: Appendix}

\begin{prop} \label{Proposition: fiber ring calculation}
	Let $K$ be a field that is finitely generated over its prime field. Let $E$ be a finite extension of $K$, 
	$v$ a discrete valuation of $K$, and $K_v$ the completion of $K$ with respect to $v$. Then there are at most 
	$[E:K]_{\operatorname{s}}$ valuations $w$ 
	extending $v$ to $E$, and if $E_w$ is the completion of $E$ with respect to the valuation $w$, 
	then there exists an 
	isomorphism of $K_v$-algebras
		\benn \label{Equation: Alg and Top isomorphism}
			K_v \otimes_K E  \cong \prod_{w|v} E_w.
		\eenn
\end{prop}

\begin{proof}
	If $E$ is a separable extension of $K$, this is proved in \cite[II.9-10]{Cassels_Frohlich}. Any
	algebraic extension can be decomposed as $K \subset \sep{E} \subset E$, where $\sep{E}$ is
	the separable closure of $K$ in $E$, and $E / \sep{E}$ is a purely inseparable extension.
	By tensoring first up to the separable closure, we may apply the result in the separable case 
	and reduce to the situation where $E / K$ is a purely inseparable extension. Thus we may 
	suppose $K$ has positive characteristic $p$. It now suffices to
	show that the valuation $v$ extends in exactly one way to $E$, and 
	that $K_v \otimes_K E \cong E_w$ holds. 
	To that end, we may even reduce to the case where $E$ is a simple nontrivial extension of $K$;
	i.e., there is $\gamma \in E \smallsetminus K$ such that $E = K(\gamma)$.

	We first argue $K_v \otimes_K E$ is a field. The valuation ring $O_v \subset K$, being the localization
	of an algebra of finite type over $\FF_p$, is a G-ring \cite[\S32]{Matsumura_CRT_1989}. Hence 
	$O_v \to O_v^{\wedge} = \Kvcirc$ is a regular homomorphism, which implies $K_v = \Frac(\Kvcirc)$ is 
	geometrically regular over $K = \Frac(O_v)$. In particular, $K_v \otimes_K E$ is a reduced ring. 
	
	On the other hand, as $E$ is a simple purely inseparable extension of $K$, we may write 
	$E = K[x] / (f(x))$ for some irreducible polynomial $f(x) = x^{p^n} - a = (x - \gamma)^{p^n}$, some positive 
	integer $n$, $a \in K$ and $\gamma \in E \smallsetminus K$. Evidently 
	$K_v \otimes_K E = K_v[x] / (f(x))$ is reduced if and only if $\gamma \not\in K_v$. Thus $f(x)$ is irreducible 
	over $K_v$ and $K_v \otimes_K E$ is a field. Note that it is the (unique) minimal extension of $K_v$ containing
	$E$.
	
	If $F$ is any finite extension of $K_v$, then $F$ inherits a unique extension of the valuation $v$ and is complete 
	with respect to the extended valuation. Therefore $K_v \otimes_K E$ is a complete field under the unique
	extension of $v$. Let $w$ be the restriction of the extended valuation to $E \subset K_v \otimes_K E$. By 
	continuity the completion $E_w$ injects canonically into $K_v \otimes_K E$, and 
	since $K_v \otimes_K E$ is the minimal extension of $K_v$ containing $E$, we must have 
	$E_w = K_v \otimes_K E$. We have already mentioned that $v$ extends uniquely to $K_v \otimes_K E$, so
	the proof is complete. 
\end{proof}


\begin{cor} \label{Corollary: Degree formula}
	Let $X$ be a variety over the function field $K$ as in previous sections.
	If $x \in |X|$ is a closed point, $v$ is a place of $K$ and $\psi: X_{K_v} \to X$ is the 
	base change morphism, then there are at most $[K(x):K]_{\operatorname{s}}$ points
	in $\psi^{-1}(x)$, and
		\[
			[K(x):K] = \sum_{y \in \psi^{-1}(x)} [K_v(y):K_v].
		\]
\end{cor}

\begin{proof}
	The ring of functions on the scheme-theoretic fiber $\psi^{-1}(x)$ is $K_v \otimes_K K(x)$. 
	Use Proposition~\ref{Proposition: fiber ring calculation} and compute dimensions over $K_v$.
\end{proof}

	\bibliographystyle{alpha}
	\bibliography{xander_bib}

\providecommand\biburl[1]{\texttt{#1}}
\begin{thebibliography}{Zha95b}

\bibitem[Bak08]{Baker_Function_Fields_2006}
Matthew Baker.
\newblock A finiteness theorem for canonical heights attached to rational maps
  over function fields.
\newblock \verb+arXiv:math/0601046v2+, to appear in \textit{J. Reine Angew.
  Math.}, 2008.

\bibitem[BD01]{Briend_Duval_2001}
Jean-Yves Briend and Julien Duval.
\newblock Deux caract\'erisations de la mesure d'\'equilibre d'un endomorphisme
  de {$\mathbb{P}^k(\mathbb{C})$}.
\newblock {\em Publ. Math. Inst. Hautes \'Etudes Sci.}, (93):145--159, 2001.

\bibitem[Ber90]{Berkovich_Spectral_Theory_1990}
Vladimir~G. Berkovich.
\newblock {\em Spectral theory and analytic geometry over non-{A}rchimedean
  fields}, volume~33 of {\em Mathematical Surveys and Monographs}.
\newblock American Mathematical Society, Providence, RI, 1990.

\bibitem[Ber94]{Berkovich_Vanishing_Cycles_Formal}
Vladimir~G. Berkovich.
\newblock Vanishing cycles for formal schemes.
\newblock {\em Invent. Math.}, 115(3):539--571, 1994.

\bibitem[BG06]{Bombieri-Gubler_2006}
Enrico Bombieri and Walter Gubler.
\newblock {\em Heights in {D}iophantine geometry}, volume~4 of {\em New
  Mathematical Monographs}.
\newblock Cambridge University Press, Cambridge, 2006.

\bibitem[BGS95]{Bloch-Gillet-Soule_Arakelov_Theory}
S.~Bloch, H.~Gillet, and C.~Soul{\'e}.
\newblock Non-{A}rchimedean {A}rakelov theory.
\newblock {\em J. Algebraic Geom.}, 4(3):427--485, 1995.

\bibitem[BH05]{Baker-Hsia_Dynamics}
Matthew~H. Baker and Liang-Chung Hsia.
\newblock Canonical heights, transfinite diameters, and polynomial dynamics.
\newblock {\em J. Reine Angew. Math.}, 585:61--92, 2005.

\bibitem[BL93]{Bosch-Lutkebohmert_Formal_Rigid_1}
Siegfried Bosch and Werner L{\"u}tkebohmert.
\newblock Formal and rigid geometry. {I}. {R}igid spaces.
\newblock {\em Math. Ann.}, 295(2):291--317, 1993.

\bibitem[BR06]{Baker_Rumely_Equidistribution_2005}
Matthew~H. Baker and Robert Rumely.
\newblock Equidistribution of small points, rational dynamics, and potential
  theory.
\newblock {\em Ann. Inst. Fourier (Grenoble)}, 56(3):625--688, 2006.

\bibitem[BR08]{Baker_Rumely_Potential_Berkovich_Line_2008}
Matthew Baker and Robert Rumely.
\newblock Potential theory on the {B}erkovich projective line.
\newblock \verb+http://www.math.gatech.edu/~mbaker/pdf/NewBerkBook_111008.pdf+,
  manuscript, 2008.

\bibitem[CF67]{Cassels_Frohlich}
J.~W.~S. Cassels and A.~Fr{\"o}hlich, editors.
\newblock {\em Algebraic number theory}.
\newblock Proceedings of an instructional conference organized by the London
  Mathematical Society (a NATO Advanced Study Institute) with the support of
  the Inter national Mathematical Union. Edited by J. W. S. Cassels and A.
  Fr\"ohlich. Academic Press, London, 1967.

\bibitem[CL06]{Chambert-Loir_Measures_2005}
Antoine Chambert-Loir.
\newblock Mesures et \'equidistribution sur les espaces de {B}erkovich.
\newblock {\em J. Reine Angew. Math.}, 595:215--235, 2006.

\bibitem[CLT08]{Chambert-Loir_Thuillier_Mahler_Integrals_2008}
Antoine Chambert-Loir and Amaury Thuillier.
\newblock Mesures de {M}ahler et \'equidistribution logarithmique.
\newblock \verb+arXiv:math/0612556v2 [math.NT]+, to appear in Ann. Inst.
  Fourier (Grenoble), 2008.

\bibitem[Con99]{Conrad_Irreducible_Components}
Brian Conrad.
\newblock Irreducible components of rigid spaces.
\newblock {\em Ann. Inst. Fourier (Grenoble)}, 49(2):473--541, 1999.

\bibitem[Fak03]{Fakhruddin_Self_Maps_2003}
Najmuddin Fakhruddin.
\newblock Questions on self maps of algebraic varieties.
\newblock {\em J. Ramanujan Math. Soc.}, 18(2):109--122, 2003.
\newblock He proves that every dynamical system embeds in projective space.

\bibitem[FRL06]{Favre-Rivera-Letelier_Quant_Equi_2006}
Charles Favre and Juan Rivera-Letelier.
\newblock \'{E}quidistribution quantitative des points de petite hauteur sur la
  droite projective.
\newblock {\em Math. Ann.}, 335(2):311--361, 2006.

\bibitem[FRL07]{Favre-Rivera-Letelier_Quant_Equi_Corrigendum}
Charles Favre and Juan Rivera-Letelier.
\newblock Corrigendum to: ``{Q}uantitative uniform distribution of points of
  small height on the projective line'' ({F}rench) [{M}ath. {A}nn. {\bf 335}
  (2006), no. 2, 311--361; mr2221116].
\newblock {\em Math. Ann.}, 339(4):799--801, 2007.

\bibitem[Ful98]{Fulton_Intersection_Theory_1998}
William Fulton.
\newblock {\em Intersection theory}, volume~2 of {\em Ergebnisse der Mathematik
  und ihrer Grenzgebiete. 3. Folge. A Series of Modern Surveys in Mathematics
  [Results in Mathematics and Related Areas. 3rd Series. A Series of Modern
  Surveys in Mathematics]}.
\newblock Springer-Verlag, Berlin, second edition, 1998.

\bibitem[Gub97]{Gubler_Heights_M_Fields_1997}
Walter Gubler.
\newblock Heights of subvarieties over {$M$}-fields.
\newblock In {\em Arithmetic geometry (Cortona, 1994)}, Sympos. Math., XXXVII,
  pages 190--227. Cambridge Univ. Press, Cambridge, 1997.

\bibitem[Gub98]{Gubler_Local_Heights_of_Subvarieties_1998}
Walter Gubler.
\newblock Local heights of subvarieties over non-{A}rchimedean fields.
\newblock {\em J. Reine Angew. Math.}, 498:61--113, 1998.

\bibitem[Gub03]{Gubler_Local_Canonical_Heights_2003}
Walter Gubler.
\newblock Local and canonical heights of subvarieties.
\newblock {\em Ann. Sc. Norm. Super. Pisa Cl. Sci. (5)}, 2(4):711--760, 2003.

\bibitem[Gub07]{Gubler_Bogomolov_2007}
Walter Gubler.
\newblock The {B}ogomolov conjecture for totally degenerate abelian varieties.
\newblock {\em Invent. Math.}, 169(2):377--400, 2007.

\bibitem[Gub08]{Gubler_FF_Equi_2008}
Walter Gubler.
\newblock Equidistribution over function fields.
\newblock \verb+arXiv:0801.4508v1+, to appear in Manuscripta Math., 2008.

\bibitem[Laz04]{Lazarsfeld_Positivity_2004}
Robert Lazarsfeld.
\newblock {\em Positivity in algebraic geometry. {I}}, volume~48 of {\em
  Ergebnisse der Mathematik und ihrer Grenzgebiete. 3. Folge. A Series of
  Modern Surveys in Mathematics [Results in Mathematics and Related Areas. 3rd
  Series. A Series of Modern Surveys in Mathematics]}.
\newblock Springer-Verlag, Berlin, 2004.
\newblock Classical setting: line bundles and linear series.

\bibitem[Mat89]{Matsumura_CRT_1989}
Hideyuki Matsumura.
\newblock {\em Commutative ring theory}, volume~8 of {\em Cambridge Studies in
  Advanced Mathematics}.
\newblock Cambridge University Press, Cambridge, second edition, 1989.
\newblock Translated from the Japanese by M. Reid.

\bibitem[Pet07]{Petsche_Elliptic_Equi_2007}
Clayton Petsche.
\newblock Non-archimedean equidistribution on elliptic curves with global
  applications.
\newblock \verb+arXiv:0710.3957v2+, preprint, 2007.

\bibitem[Sil07]{Silverman_Dynamics_Book_2007}
Joseph~H. Silverman.
\newblock {\em The arithmetic of dynamical systems}, volume 241 of {\em
  Graduate Texts in Mathematics}.
\newblock Springer, New York, 2007.

\bibitem[SUZ97]{Szpiro-Ullmo-Zhang}
L.~Szpiro, E.~Ullmo, and S.~Zhang.
\newblock \'{E}quir\'epartition des petits points.
\newblock {\em Invent. Math.}, 127(2):337--347, 1997.

\bibitem[Yua08]{Xinyi_Arithmetic_Bigness_2006}
Xinyi Yuan.
\newblock Big line bundles over arithmetic varieties.
\newblock {\em Invent. Math.}, 173(3):603--649, 2008.

\bibitem[Zha95a]{Zhang_Positive_Arithmetic_Varieties_1995}
Shou-wu Zhang.
\newblock Positive line bundles on arithmetic varieties.
\newblock {\em J. Amer. Math. Soc.}, 8(1):187--221, 1995.

\bibitem[Zha95b]{Zhang_Small_Points_1995}
Shou-wu Zhang.
\newblock Small points and adelic metrics.
\newblock {\em J. Algebraic Geom.}, 4(2):281--300, 1995.

\end{thebibliography}

\end{document}